\pdfoutput=1 
\documentclass[11pt,a4paper,twoside,draft,leqno]{article}
\usepackage{a4wide}
\usepackage{xcolor}
\usepackage{amsmath}
\usepackage{amssymb}
\usepackage{amsthm}
\usepackage{mathrsfs}
\usepackage[T1]{fontenc}
\usepackage[utf8x]{inputenc}
\usepackage{url}
\usepackage{esint}
\usepackage{enumitem}
\setlist[enumerate,1]{label=(\alph*), ref=(\alph*)}
\usepackage{hyperref}
\usepackage{mathtools}
\mathtoolsset{showonlyrefs}
\usepackage{dsfont}
\usepackage{phonetic}
\usepackage{extpfeil}
\usepackage[all]{xy}

\newtheorem{introtheorem}{Theorem}

\swapnumbers
\theoremstyle{plain}
\newtheorem{theorem}{Theorem}[section]
\newtheorem{lemma}[theorem]{Lemma}

\newtheorem{corollary}[theorem]{Corollary}

\theoremstyle{remark}
\newtheorem{remark}[theorem]{Remark}
\newtheorem*{remark*}{Remark}
\newtheorem{example}[theorem]{Example}

\theoremstyle{definition}
\newtheorem{definition}[theorem]{Definition}

\newcommand{\ccspace}[1]{\mathscr{K}(#1)}

\newcommand{\dspace}[2]{\mathscr{D}(#1,#2)}

\newcommand{\project}[1]{{#1}_\natural}

\newcommand{\perpproject}[1]{#1_\natural^\perp}

\newcommand{\grass}[2]{\mathbf{G}(#1,#2)}

\newcommand{\measureball}[2]{{#1}\,{#2}}

\newcommand{\Var}[1]{\mathbf{V}_{#1}}     
\DeclareMathOperator{\VarTan}{VarTan}     
\newcommand{\var}[1]{\mathbf{v}_{#1}}     

\newcommand{\adim}{n}
\newcommand{\vdim}{d}

\newcommand{\oball}[2]{\mathbf{U}(#1,#2)}

\newcommand{\cball}[2]{\mathbf{B}(#1,#2)}

\newcommand{\sphere}[1]{\mathbb{S}^{#1}}

\newcommand{\textint}[2]{{\textstyle\int_{#1}^{#2}}}

\newcommand{\nat}{\mathbb{N}}
\newcommand{\natp}{\mathscr{P}}

\newcommand{\integers}{\mathbf{Z}}

\newcommand{\adm}[1]{\mathfrak{D}({#1})}

\newcommand{\R}{\mathbf{R}}

\newcommand{\CF}{\mathds{1}}

\newcommand{\Sh}{\mathbf{H}}

\newcommand{\HM}{\mathscr{H}}

\newcommand{\HD}{{d_{\HM}}}

\newcommand{\HDK}[1]{{d_{\HM,#1}}}

\newcommand{\density}{\boldsymbol{\Theta}}

\newcommand{\eqclsl}{\boldsymbol{[}}
\newcommand{\eqclsr}{\boldsymbol{]}}

\newcommand{\unitmeasure}[1]{\boldsymbol{\alpha}(#1)}

\newcommand{\restrict}{ \mathop{ \rule[1pt]{.5pt}{6pt} \rule[1pt]{4pt}{0.5pt} }\nolimits }

\newcommand{\ud}{\ensuremath{\,\mathrm{d}}}

\newcommand{\uD}{\ensuremath{\mathrm{D}}}

\newcommand{\id}[1]{\mathrm{id}_{#1}}

\newcommand{\lIm}{[}
\newcommand{\rIm}{]}

\newcommand{\scale}[1]{\boldsymbol{\mu}_{#1}}
\newcommand{\trans}[1]{\boldsymbol{\tau}_{#1}}

\newcommand{\tbcup}{\mathop{{\textstyle \bigcup}}}

\newcommand{\Clos}[1]{\mathop{\mathrm{Clos}}#1}

\newcommand{\side}[1]{\mathbf{l}(#1)}
\newcommand{\centre}[1]{\mathbf{c}(#1)}
\newcommand{\vertex}[1]{\mathbf{o}(#1)}

\newcommand{\cInt}[1]{\operatorname{Int}_{\mathrm{c}}(#1)}
\newcommand{\cBdry}[1]{\partial_{\mathrm{c}}#1}

\newcommand{\cubes}{\mathbf{K}}

\newcommand{\CX}{\mathbf{CX}}

\newcommand{\Dirac}[1]{\operatorname{Dirac}(#1)}

\newcommand{\VF}{\mathscr{X}}

\DeclareMathOperator{\Hom}{Hom}

\newcommand{\Bdry}[1]{\partial{}{#1}}

\DeclareMathOperator{\ap}{ap}

\DeclareMathOperator{\lin}{span}

\newcommand{\cnt}{\mathscr{C}}

\newcommand{\orthproj}[2]{\mathbf{O}^\ast({#1},{#2})}

\newcommand{\orthgroup}[1]{\mathbf{O}({#1})}

\DeclareMathOperator{\Int}{Int}

\DeclareMathOperator{\spt}{spt}

\DeclareMathOperator{\Tan}{Tan}

\DeclareMathOperator{\Lip}{Lip}

\DeclareMathOperator{\dist}{dist}

\DeclareMathOperator{\diam}{diam}

\DeclareMathOperator{\without}{\sim}

\DeclareMathOperator{\im}{im}

\DeclareMathOperator{\FAUE}{AUE}
\DeclareMathOperator{\FAE}{AE}
\DeclareMathOperator{\FBC}{BC}
\DeclareMathOperator{\FwBC}{wBC}
\DeclareMathOperator{\FAC}{AC}

\date{}

\title{Equivalence of the ellipticity conditions for geometric variational problems.}

\author{
  Antonio De Rosa
  \and
  S{\l}awomir Kolasi{\'n}ski
}

\begin{document}

\maketitle


\begin{abstract}
    We exploit the so called \emph{atomic condition}, recently defined in \cite[Comm. Pure
    Appl. Math.]{DDG2016rect} and proved to be necessary and sufficient for the
    validity of the anisotropic counterpart of the Allard rectifiability
    theorem. In~particular, we address an open question of this seminal work,
    showing that the atomic condition implies the strict Almgren geometric
    ellipticity condition.
\end{abstract}

\section{Introduction}
Since the pioneering works of Almgren \cite{Alm68,Alm76}, a deep effort has been
devoted to the understanding of ellptic integrands in geometric variational
problems. In particular, Almgren introduced the class of elliptic geometric
integrands (\cite[IV.1(7)]{Alm76} or~\cite[1.6(2)]{Alm68}), further denoted
$\FAUE$, which allowed him to prove regularity for minimisers in~\cite{Alm68}.
 
Very recently, an ongoing interest on the anisotropic Plateau problem has lead
to a series of reformulations and results in this direction,
see~\cite{HP2016gm,DDG2016,DDG2017,DDG2017b, DeR2016,FangKol2017}.
In~particular, in~\cite{DDG2016rect} (see also Definition~\ref{def:AC}) a~new
ellipticity condition, called the~\emph{atomic condition}, further denoted~$\FAC$,
has been introduced and proved to be necessary and sufficient to get an Allard type
rectifiability result for varifolds whose anisotropic first variation is a~Radon
measure. The authors can prove that, in co-dimension~one and in dimension~one,
$\FAC$~is equivalent to the strict convexity of the integrand.

For general co-dimension there is no understanding of~$\FAC$ in the~literature
and this is~stated as an open problem in~\cite[Page 2]{DDG2016rect}:
\begin{quote}
    ``Since the atomic condition $\FAC$ is essentially necessary to the validity of
    the rectifiability theorem, it is relevant to relate it to the previous
    known notions of \emph{ellipticity} (or \emph{convexity}) of $F$ with
    respect to the ``plane'' variable. This task seems to be quite hard in the
    general case.''
\end{quote}
The aim of this paper is to address this open question, comparing
condition~$\FAC$ with the classical notion of geometric ellipticity introduced
by~Almgren. 

We~present for the moment an informal version of our
main result, see~\ref{thm:wBC-in-AE}:
\begin{introtheorem}
    \label{roughth}
    If a $\cnt^1$~integrand satisfies the atomic condition at some point $x \in
    \R^{\adim}$, then it also satisfies the strict Almgren ellipticity
    condition at~$x$; see~\ref{thm:wBC-in-AE}.
\end{introtheorem}
In~particular, if the co-dimension equals one, then strict convexity of the
integrand implies the strict Almgren ellipticity. Moreover in higher
co-dimension, our work paves the way to construct anisotropic functionals
satisfying the Almgren ellipticity condition. Indeed, although the theory of
existence and regularity for minimizers has been actively developed in the
literature, there are essentially no examples (in higeher co-dimension) of
Almgren elliptic integrands, beside the perturbations of the area functional.

It is worth to remark that there is no hope of improving Theorem~\ref{roughth}
showing that the atomic condition implies the uniform Almgren ellipticity
condition, see Remark~\ref{remarkimpl}. Indeed, if this was the case, in
co-dimension one the strict convexity of the integrand (which is equivalent to
the atomic condition) would imply the uniform Almgren ellipticity, which in~turn
implies the uniform convexity, leading to a contradiction.

In order to prove Theorem~\ref{roughth}, we need to get several auxiliary
results of independent interest. In~particular, in Section~\ref{notion} we
introduce another ellipticity condition for integrands, named~$\FBC$, and in
Section~\ref{equiv} we prove that it is equivalent to~$\FAC$; see
Definition~\ref{def:BC} and Lemma~\ref{lem:AC-eq-BC}. $\FBC$ has the advantage
of being more geometric than the algebraic condition~$\FAC$, thus providing
a~useful tool not only for the proof of Theorem~\ref{roughth}, but also for
future further understanding of the atomic condition.
In~Section~\ref{sec:rect-test-pairs} we show that the original Almgren
ellipticity condition~\cite[IV.1(7)]{Alm76} is the same as the condition used
in~\cite[3.16]{FangKol2017} which involves unrectifiable surfaces; see
Corollary~\ref{cor:mae:equiv}. To this end we provide a~deformation
theorem~\ref{thm:mae:dt} which preserves unrectifiability of the unrectifiable
part of a~given set; see Theorem~\ref{thm:mae:dt}.  Moreover, in
Section~\ref{Plateau}, Theorem~\ref{thm:wbc-exist}, we~provide an independent
proof of the existence of minimisers of anisotropic energies satisfying~$\FAC$
(or equivalently $\FBC$), improving the recent solutions to the set theoretical
approach to the anisotropic Plateau
problem~\cite{DDG2017b,FangKol2017}. Gathering these results, we provide in
Section~\ref{sec:wbc-in-ae} the proof of Theorem~\ref{roughth}, see
Theorem~\ref{thm:wBC-in-AE}.

The last crucial point is that the proof of Theorem~\ref{roughth} in
Section~\ref{sec:wbc-in-ae} requires the validity of a~seemingly harmless
property: the class of compact sets~$X$ used by Almgren to test the strict
ellipticity considition (see \cite[IV.1(7)]{Alm76} or~\cite[1.6(2)]{Alm68}) is
closed under gluing together many rescaled copies of~$X$;
see~\ref{def:good-test-pairs}. In~\ref{thm:all-tp-good} we show indeed that this
property is true, but our proof is quite complicated and employs some
sophisticated tools of algebraic topology; see also the introduction to
Section~\ref{sec:all-test-pairs-good}. Giving it some thought, Almgren's
condition that~$X$ cannot be retracted onto its boundary sphere is topological
in nature, so it is reasonable that topological arguments are
indispensable. Moreover, the existence of the Adams' surface, which is
retractible onto its boundary and is obtained by gluing together two surfaces
that cannot be retracted onto their respective boundaries, supports the claim
that the proof of Almgren's class being closed under the gluing operation is
highly non-trivial; see~\ref{rem:ctp-good-p}. This question is fully addressed
in Section~\ref{sec:all-test-pairs-good}.

\section{Notation}
For the whole article we fix two integers~$\vdim$ and~$\adim$
  satisfying $2 \le \vdim \le \adim$.

In~principle we shall follow the notation of Federer;
see~\cite[pp.~669--671]{Fed69}. In~particular, given two sets $A,B$, we denote
with $A \without B$ their \emph{set-theoretic difference} and, for every $a\in
\R^\adim$ and~$s \in \R$ we define the functions $\trans{a}(x) = a+x$
and~$\scale{s}(x) = sx$; see~\cite[2.7.16, 4.2.8]{Fed69}. Concerning varifolds,
we shall follow Allard~\cite{All72}.

Following~\cite{Alm68} and~\cite{Alm00}, if $S \in \grass{\adim}{\vdim}$ is
a~$\vdim$~dimensional linear subspace of~$\R^{\adim}$, then $\project S \in
\Hom(\R^{\adim},\R^{\adim})$ shall denote the \emph{orthogonal projection}
onto~$S$. In~particular, if $p \in \orthproj{\adim}{\vdim}$ is such that $\im
p^* = S$, then $\project S = p^* \circ p$.

We divert in notation from~\cite{Fed69} in the following ways. To denote the
\emph{image of a set} $A \subseteq X$ under some map $f : X \to Y$ (more
generally, under a relation $f \subseteq X \times Y$) we always use square
brackets: $f \lIm A \rIm$. We employ the symbol $\id{X}$ to denote the
\emph{identity map} $X \to X$ and $\CF_A$ to denote the \emph{characteristic
  function} $X \to \{0,1\}$ of~$A \subseteq X$. We also use abbreviations for
\emph{intervals}, e.g., $(a,b] = \{ t : a < t \le b \}$. Moreover, we~denote
with~$\nat$ the set of \emph{non-negative integers}, i.e., $\nat = \natp \cup
\{0\}$. If $(X,\rho)$ is a~metric space, $A \subseteq X$, and $x \in X$, then we
define $\dist(x,A) = \inf \rho \lIm A \times \{x\} \rIm$. We sometimes write $X
\hookrightarrow Y$, $X \twoheadrightarrow Y$, or $X \xrightarrow{\simeq} Y$ to
emphasis that a map is \emph{injective}, \emph{surjective}, or \emph{bijective}
respectively. We denote with~$\Bdry A$ the \emph{topological boundary} of a
set~$A$. Whenever $A$, $B$ are subsets of a vector space we write $A+B$ to
denote the \emph{algebraic sum} of~$A$ and~$B$, i.e., $A+B = \{ a+b : a \in A
\,,\, b \in B \}$; in~particular, if $\varepsilon \in (0,\infty)$, then $A +
\cball{0}{\varepsilon}$ is the \emph{$\varepsilon$-thickening} of~$A$. If $R$ is
a~ring and~$A$, $B$ are $R$-modules, then $A \oplus B$ denotes their
\emph{direct sum}; cf.~\cite[Chap.~V, Def.~5.6]{ES1952}. For $a,b \in \natp$ the
symbol $\gcd(a,b)$ denotes the \emph{greatest common divisor} of~$a$ and~$b$ and
$a \operatorname{mod} b$ means the \emph{remainder} of the division of~$a$
by~$b$.

In Sections~\ref{sec:wbc-in-ae} and~\ref{sec:all-test-pairs-good} we shall need
to use tools of algebraic topology. We shall work in the category of all pairs
of topological spaces {$\text{{\Large \vara}}_1$} as defined in~\cite[Chap.~I,
\S{1}, p.~5]{ES1952}. We~write $\Sh_k(X,A;G)$ and $\Sh^k(X,A;G)$ for the
$k^{\mathrm{th}}$ \emph{singular homology} and \emph{cohomology groups} of the
pair $(X,A)$ with coefficients in~$G$; see~\cite[Chap.~VII,
Definition~2.9]{ES1952}. If $G = \integers$, then we omit~$G$ in the
notation. Similarly, if $A = \varnothing$, we omit~$A$. Given two maps $f,g : X
\to Y$ between topological spaces we write $f \approx g$ to express that~$f$
and~$g$ are \emph{homotopic}, i.e., there exists a~continuous map $h : [0,1]
\times X \to Y$ such that $h(0,\cdot) = f$ and $h(1,\cdot) = g$. If~$X$ and~$Y$
are topological spaces which are \emph{homotopy equivalent} we write $X \approx
Y$ and if they are \emph{homeomorphic} we write $X \simeq Y$.

\begin{definition}[\protect{cf.~\cite[Chap.~XI, Def.~4.1]{ES1952}}]
    \label{def:degree}
    Let $B \subseteq \R^{\adim}$ be homeomorphic to the standard $k$-dimensional
    sphere and $f : B \to B$ be continuous. Suppose $\sigma$ is the generator of
    the $k^{\mathrm{th}}$~homology group~$\Sh_{k}(B)$ of~$B$ and $f_* :
    \Sh_{k}(B) \to \Sh_{k}(B)$ is the map induced by~$f$. The \emph{topological
      degree $\deg (f) \in \integers$ of~$f$} is the unique integer such that
    $f_*(\sigma) = \deg (f) \cdot \sigma$.
\end{definition}

\section{Basic definitions}
\begin{definition}[\protect{cf.~\cite[1.2]{Alm68}}]
    \label{def:Ck-integrand}
    A function $F : \R^{\adim} \times \grass{\adim}{\vdim} \to (0,\infty)$ of
    class~$\cnt^k$ for some non-negative integer~$k$ is called
    a~\emph{$\cnt^k$~integrand}.

    If $\inf \im F / \sup \im F \in (0,\infty)$, then we say that $F$ is
    \emph{bounded}.
\end{definition}

\begin{definition}[\protect{cf.~\cite[3.1]{Alm68}}]
    \label{def:pull-back}
    If $\varphi \in \cnt^{1}(\R^{\adim},\R^{\adim})$ and $F$ is an integrand, then 
    the \emph{pull-back} integrand~$\varphi^\#F$ is given by
    \begin{displaymath}
        \varphi^\#F(x,T) = \left\{
            \begin{aligned}
                &F\bigl( \varphi(x), \uD\varphi(x)\lIm T \rIm \bigr) 
                \| {\textstyle \bigwedge_{\vdim}} \uD\varphi(x) \circ \project{T} \|
                && \text{if } \dim \uD\varphi(x)\lIm T \rIm = \vdim
                \\
                &0 
                && \text{if } \dim \uD\varphi(x)\lIm T \rIm < \vdim \,.
            \end{aligned}
        \right.
    \end{displaymath}
    If $\varphi$ is a diffeomorphism, then the \emph{push-forward} integrand is
    given by $\varphi_\#F = (\varphi^{-1})^\#F$.
\end{definition}

\begin{definition}[\protect{cf.~\cite[1.2]{Alm68}}]
    \label{def:Fx-integrand}
    If $F$ is a $\cnt^k$~integrand and $x \in \R^{\adim}$, then we define
    the \emph{frozen} $\cnt^k$~integrand $F^x$ by the formula
    \begin{displaymath}
        F^x(y,S) = F(x,S)
        \quad \text{for every $y \in \R^{\adim}$ and $S \in \grass{\adim}{\vdim}$} \,.
    \end{displaymath}
\end{definition}

\begin{remark}
    \label{rem:frozen-bounded}
    Since $F : \R^{\adim} \times \grass{\adim}{\vdim} \to (0,\infty)$ and
    $\grass{\adim}{\vdim}$ is compact, it follows that for any $x \in
    \R^{\adim}$ the frozen integrand $F^x$ is bounded.
\end{remark}

\begin{definition}
    We say that $S \subseteq \R^{\adim}$ is a \emph{$\vdim$-set} if $S$ is
    $\HM^{\vdim}$~measurable and $\HM^{\vdim}(S \cap K) < \infty$ for any
    compact set $K \subseteq \R^{\adim}$.
\end{definition}

\begin{definition}
    \label{def:R-U}
    Assume $S \subseteq \R^{\adim}$ is a $\vdim$-set. We define
    \begin{displaymath}
        \mathcal R(S) = \{ x \in S : \density^{\vdim}(\HM^{\vdim} \restrict S,x) = 1 \} 
        \quad \text{and} \quad
        \mathcal U(S) = S \without \mathcal R(S) \,.
    \end{displaymath}
\end{definition}

\begin{remark}
    \label{rem:R-U}
    Observe that $\density^{\vdim}(\HM^{\vdim} \restrict S,\cdot)$ is a Borel
    function, so $\mathcal R(S)$ is $\HM^{\vdim}$~measurable.
    Employing~\cite{Mattila1975} and~\cite[2.9.11]{Fed69}, we observe that
    $\mathcal R(S)$ is countably $(\HM^{\vdim},\vdim)$~rectifiable and $\mathcal
    U(S)$ is purely $(\HM^{\vdim},\vdim)$~unrectifiable.
\end{remark}

\begin{remark}
    Recall that $\boldsymbol{\gamma}_{\adim,\vdim}$ denotes the canonical
    probability measure on $\grass{\adim}{\vdim}$ invariant under the action of
    the orthogonal group $\orthgroup{\adim}$, also called Haar measure; see~\cite[2.7.16(6)]{Fed69}.
\end{remark}

\begin{definition}[\protect{cf.~\cite[3.5]{All72}}]
    \label{def:var-S}
    Assume $S \subseteq \R^{\adim}$ is a $\vdim$-set. We define $\var{\vdim}(S)
    \in \Var{\vdim}(\R^{\adim})$ by setting for every $\alpha \in
    \ccspace{\R^{\adim} \times \grass{\adim}{\vdim}}$
    \begin{equation*}
        \var{\vdim}(S)(\alpha)
        = \int_{\mathcal R(S)} \alpha(x, \Tan^{\vdim}(\HM^{\vdim} \restrict \mathcal R(S), x)) \ud \HM^{\vdim}(x)
        + \int_{\mathcal U(S)} \int \alpha(x, T) \ud \boldsymbol{\gamma}_{\adim,\vdim}(T) \ud \HM^{\vdim}(x)\,.
    \end{equation*}
\end{definition}

\begin{definition}
    \label{def:Phi-F}
    If $F$ is a $\cnt^k$~integrand, we define the functional $\Phi_F :
    \Var{\vdim}(\R^{\adim}) \to [0,\infty]$ by the formula
    \begin{displaymath}
        \Phi_F(V) = \int F(x,S) \ud V(x,S) \,.
    \end{displaymath}
\end{definition}

\begin{remark}
    \label{rem:cpt-spt-vari}
    If $\spt \|V\|$ is compact we have $\Phi_F(V) = V(\gamma F)$ for any $\gamma
    \in \dspace{\R^{\adim}}{\R}$ satisfying $\spt \|V\| \subseteq \gamma^{-1} \{
    1 \}$.
\end{remark}

\begin{definition}
    \label{def:Phi-F-on-S}
    If $S \subseteq \R^{\adim}$ is a $\vdim$-set, then we define $\Phi_F(S) = \Phi_F(\var{\vdim}(S))$ and
    $$\Psi_{F}(S) = \Phi_{F}(S)
        + \int_{\mathcal U(S)} \bigl( \sup \im F^x - \textint{}{} F(x,T) \ud \boldsymbol{\gamma}_{\adim,\vdim}(T) \bigr)\ud \HM^{\vdim}(x)\, .
   $$
    For any other subset $S$ of $\R^{\adim}$, we define $\Psi_F(S) = \Phi_F(S) = \infty$.
\end{definition}

\begin{remark}
    \label{rem:pull-back}
    Assume $V \in \Var{\vdim}(\R^{\adim})$, $\varphi : \R^{\adim} \to
    \R^{\adim}$ is of class~$\cnt^1$, and $F$ is a~$\cnt^0$~integrand. Then
    \begin{displaymath}
        \Phi_{\varphi^{\#}F}(V) = \Phi_F(\varphi_{\#}V) \,.
    \end{displaymath}
    If $S \subseteq \R^{\adim}$ is a $\vdim$-set, then
    \begin{displaymath}
        \varphi_{\#}\var{\vdim}(S) = \var{\vdim}(\varphi \lIm S \rIm) 
    \end{displaymath}
    in the case $\varphi$ is injective and $S$ is countably
    $(\HM^\vdim,\vdim)$~rectifiable, or in the case $\varphi = \scale{r}$ for some $r \in
    (0,\infty)$, or in the case $\varphi = \trans{a}$ for some $a \in \R^{\adim}$.
\end{remark}

\begin{remark}
    \label{rem:Psi-Phi}
    If $S$ is a $\vdim$-set, $F$ is a $\cnt^0$~integrand and $x \in \R^{\adim}$, then
    \begin{displaymath}
        \Psi_{F^x}(S) = \Phi_{F^x}(\mathcal R(S)) + \HM^{\vdim}(\mathcal U(S)) \sup \im F^x \,.
    \end{displaymath}
\end{remark}

\begin{definition}
    For any set $X$ and an element $x \in X$ we denote by $\Dirac x$ the measure
    over~$X$ with a single atom at~$x$, i.e.,
    \begin{displaymath}
        \Dirac x(A) =
        \left\{
            \begin{aligned}
                &1 &&\text{if $x \in A$} \,,
                \\
                &0 &&\text{if $x \notin A$} \,,
            \end{aligned} 
        \right.
        \quad \text{for $A \subseteq X$} \,.
    \end{displaymath}
    The choice of~$X$ shall always be clear from the context.
\end{definition}

\begin{definition}[\protect{cf.~\cite[4.9]{All72}}]
    Assume $U \subseteq \R^{\adim}$ is open, $V \in \Var{\vdim}(U)$, $F$ is
    a~$\cnt^1$ integrand. We define the \emph{first variation of~$V$ with
      respect to~$F$} to be the linear map $\delta_F V : \VF(U) \to \R$ given by
    the formula
    \begin{displaymath}
        \delta_F V(g) = \left. \frac{d}{dt}\right|_{t=0} \Phi_F \bigl( (\varphi_t)_{\#}V \bigr)  \,,
    \end{displaymath}
    where $g \in \VF(U)$ is a smooth compactly supported vectorfield in~$U$ and
    $\varphi_t(x) = x + tg(x)$ for $x \in U$ and $t$ in some neighbourhood
    of~$0$ in~$\R$.
\end{definition}

\begin{remark}
    \label{rem:fst-ani-var}
    Note that if $T \in \grass{\adim}{\vdim}$ and
    \begin{displaymath}
        \mathcal{G}_{\adim,\vdim} = \bigl\{ \project P : P \in \grass{\adim}{\vdim} \bigr\} \subseteq \Hom(\R^{\adim},\R^{\adim}) \,,
    \end{displaymath}
    then
    \begin{displaymath}
        A \in \Tan(\mathcal{G}_{\adim,\vdim},\project T)
        \quad \iff \quad
        A^* = A \,,
        \quad
        \project T \circ A \circ \project T = 0 \,,
        \quad \text{and} \quad
        \perpproject T \circ A \circ \perpproject T = 0 \,.
    \end{displaymath}
    For $x \in \R^\adim$ and $T \in \grass{\adim}{\vdim}$ define
    \begin{gather}
        F_T : \R^{\adim} \to \R 
        \quad \text{and} \quad
        F_x : \mathcal G_{\adim,\vdim} \to \R 
        \quad
        \text{by setting} \quad
        F_T(x) = F(x,T) = F_x(\project T) \,.
    \end{gather}
    In~\cite{DDG2016rect} the authors computed
    \begin{displaymath}
        \delta_F V(g) = \int \bigl\langle g(x) , \uD F_T(x) \bigr\rangle
        + B_F(x,T) \bullet \uD g(x) \ud V(x,T) \,,
    \end{displaymath}
    where $B_F(x,T) \in \Hom(\R^{\adim},\R^{\adim})$ is characterised by
    \begin{displaymath}
        B_F(x,T) \bullet L = F(x,T) \project T \bullet L 
        + \bigl\langle
        \perpproject{T} \circ L \circ \project{T} + (\perpproject{T} \circ L \circ \project{T})^*, \uD F_x(\project T)
        \bigr\rangle \,,
    \end{displaymath}
    whenever $L \in \Hom(\R^{\adim},\R^{\adim})$.
\end{remark}

\section{Notions of ellipticity}\label{notion}
In this section we recall the notions of ellipticity we will work with. 
\begin{definition}
    \label{def:test-pair}
    We say that $(S,D)$ is a \emph{test pair} if there exists $T \in
    \grass{\adim}{\vdim}$ such that
    \begin{gather*}
        D = T \cap \cball 01 \,,
        \quad
        B = T \cap \Bdry{\cball 01} \,,
        \quad
        \text{$S \subseteq \R^\adim$ is compact} \,,
        \quad
        \HM^{\vdim}(S) < \infty \,,
        \\
        f \lIm S \rIm \ne B
        \quad
        \text{for all $f : \R^\adim \to \R^\adim$ satisfying $\Lip f < \infty$ and $f(x) = x$ for every $x \in B$} \,.
    \end{gather*}
    We say that $(S,D)$ is a \emph{rectifiable test pair} if, in~addition, $S$
    is~$(\HM^{\vdim},\vdim)$~rectifiable.
\end{definition}

\begin{remark}
    Using a standard extension procedure for Lipschitz functions
    (e.g.~\cite[3.1.1, Theorem~1]{EvansGariepy1992}), one sees that the last
    condition in Definition~\ref{def:test-pair} means exactly that~$B$ is not
    a~Lipschitz retract of~$S$.
\end{remark}

\begin{example}
    \label{ex:mobius-strip}
    Let $\adim = 3$, $\vdim = 2$, $T = \R^2 \times \{0\}$, $D = T \cap \cball
    01$, and $S$ be a smoothly embedded M{\"o}bius strip with boundary $B = T
    \cap \Bdry{\cball 01}$. Observe, that $S$ itself has the homotopy type of
    a~$1$-dimensional circle because a~M{\"o}bius strip can easily be retracted
    onto the ``middle circle''. However, the inclusion $j : B \hookrightarrow S$
    has topological degree~$2$, so given any continuous map $f : S \to B$ we
    have $j \circ f = f|_B : B \to B$ and we see that
    $\deg(f|_B) = \deg(j)
    \deg(f)$ is an \emph{even} integer which means that $f|_B$
    cannot equal the identity on~$B$. Therefore, $(S,D)$ is a~rectifiable test
    pair.
\end{example}

\begin{lemma}
    \label{lem:ctp-HD-limit}
    Let $(S,D)$ be a pair of compact sets in~$\R^{\adim}$ with $\HM^{\vdim}(S) <
    \infty$ and $\{ (S_i,D_i) : i \in \nat \}$ be a~sequence of test pairs such
    that
    \begin{displaymath}
        \lim_{i \to \infty} \HD(S_i,S) = 0
        \quad \text{and} \quad
        \lim_{i \to \infty} \HD(D_i,D) = 0 \,.
    \end{displaymath}
    Then $(S,D)$ is a test pair.
\end{lemma}

\begin{proof}
    For every $i \in \nat$, let $T_i \in \grass{\adim}{\vdim}$ be such that $D_i = T_i
    \cap \cball 01$ and set $B_i = T_i \cap \Bdry{\cball 01}$. First note that
    since $\{ D_i : i \in \nat \}$ is a Cauchy sequence with respect to the
    Hausdorff metric on compact sets, we obtain that $\{ T_i : i \in \nat \}$ is
    a~Cauchy sequence in $\grass{\adim}{\vdim}$ and there exists $T \in
    \grass{\adim}{\vdim}$ such that $D = T \cap \cball 01$. Set $B = T \cap
    \Bdry{\cball 01}$.

    Assume, by contradiction, that there exists $f : \R^{\adim} \to \R^{\adim}$
    such that $\Lip f < \infty$, $f(x) = x$ for every $x \in B$, and $f \lIm S \rIm =
    B$. Set $\delta = (\Lip f)^{-1} \in (0,1]$. Then
    \begin{displaymath}
        f \lIm S + \cball 0r \rIm \subseteq B + \cball 0{r/\delta}
        \quad \text{for $r \in (0,\infty)$} \,.
    \end{displaymath}
    Choose $i \in \nat$ such that
    \begin{displaymath}
        S_i \subseteq S + \cball 0{2^{-5}\delta^2} 
        \quad \text{and} \quad
        B \subseteq B_i + \cball 0{2^{-5}\delta} \,.
    \end{displaymath}
    Then,
    \begin{displaymath}
        f \lIm S_i \rIm \subseteq B + \cball 0{2^{-5} \delta} \subseteq B_i + \cball 0{2^{-4}\delta} \,.
    \end{displaymath}
    Define $g : S_i \to B_i$ by
    \begin{gather*}
        g(y) = f(y)
        \quad \text{for } y \in S_i \without \bigl( B_i + \cball 0{2^{-4}\delta} \bigr) \,,
        \\
        g(y) = 2^{4} \delta^{-1} \dist(y,B_i) ( f(y) - y ) + y
        \quad \text{for } y \in S_i \cap \bigl( B_i + \cball 0{2^{-4}\delta} \bigr) \,.
    \end{gather*}
    For any $y \in S_i$ with $\dist(y,B_i) \le 2^{-4} \delta$ we can find $x
    \in B_i$ and $z \in B$ such that $|x-y| \le 2^{-4} \delta$ and $|x-z| \le
    2^{-5}\delta$; hence, $|y-z| \le 2^{-3}\delta$ and
    \begin{multline*}
        \dist(g(y),B_i) 
        \le |g(y)-x| 
        \le 2^{4} \delta^{-1} \dist(y,B_i) | f(y) - y | + |y-x|
        \\
        = |f(y) - f(z) + z - y| + |y-x|
        \le \delta^{-1} |y-z| + |z-y| + |y-x|
        \le 2^{-1} \,.
    \end{multline*}
    This shows that $g\lIm S_i \rIm \subseteq B_i + \cball
    0{2^{-1}}$. Composing $g$ with a Lipschitz map retracting $B_i + \cball
    0{2^{-1}}$ onto $B_i$ yields a Lipschitz retraction of~$S_i$ onto~$B_i$ and
    a~contradiction.
\end{proof}

\begin{definition}
    \label{def:AE}
    Let $x \in \R^{\adim}$ and $\mathcal P$ be a~set of pairs of compact $d$-sets
    in~$\R^\adim$.
    \begin{enumerate}
    \item \emph{Almgren uniform ellipticity with respect to~$\mathcal P$}: The
        class $\FAUE_x(\mathcal P)$ is defined to contain all
        $\cnt^0$~integrands $F$ for which there exists $c > 0$ such that for all
        $(S,D) \in \mathcal P$ there holds
        \begin{displaymath}
            \Psi_{F^x}(S) - \Psi_{F^x}(D) \ge c \bigl( \HM^{\vdim}(S) - \HM^{\vdim}(D) \bigr) \,.
        \end{displaymath}
    \item \emph{Almgren strict ellipticity with respect to~$\mathcal P$}: The
        class $\FAE_x(\mathcal P)$ is defined to contain all $\cnt^0$~integrands
        $F$ such that for all $(S,D) \in \mathcal P$ satisfying $\HM^{\vdim}(S) > \HM^{\vdim}(D)$ there holds
        \begin{displaymath}
            \Psi_{F^x}(S) - \Psi_{F^x}(D) > 0 \,.
        \end{displaymath}
    \end{enumerate}
\end{definition}

\begin{remark}
    \label{rem:test-pairs}
    \begin{enumerate}
    \item If all elements of $\mathcal P$ are pairs of $(\HM^d,d)$~rectifiable
        sets, then one can replace all occurrences of~$\Psi_{F^x}$
        with~$\Phi_{F^x}$.

    \item If $\mathcal P = \varnothing$, then $\FAE_x(\mathcal P) =
        \FAUE_x(\mathcal P)$ is the set of all $\cnt^0$~integrands.

    \item If $\mathcal P$ is the set of \emph{rectifiable} test pairs, then $F
        \in \FAUE_x(\mathcal P)$ if and only if $F$ is elliptic at~$x$ in the
        sense of~\cite[IV.1(7)]{Alm76}.

    \item If $\mathcal P$ is the set of \emph{all} test pairs, then $F \in
        \FAUE_x(\mathcal P)$ if and only if $F$ is elliptic at~$x$ in the sense
        of~\cite[3.16]{FangKol2017}.

    \end{enumerate}
\end{remark}

\begin{definition}[\protect{cf.~\cite[Definition~1.1]{DDG2016rect}}]
    \label{def:AC}
    Let $x \in \R^\adim$. The class~$\FAC_x$ is defined to contain all~$\cnt^1$
    integrands $F$ satisfying the \emph{atomic condition} at~$x$, i.e., for any
    Radon probability measure~$\mu$ over~$\grass{\adim}{\vdim}$, setting
    \begin{displaymath}
        A_x(\mu) = \int B_F(x,T) \ud \mu(T) \in \Hom(\R^{\adim},\R^{\adim}), 
    \end{displaymath}
    there holds
    \begin{enumerate}
    \item
        \label{i:ac:a}
        $\dim \ker A_x(\mu) \le \adim - \vdim$;
    \item
        \label{i:ac:b}
        if $\dim \ker A_x(\mu) = \adim - \vdim$, then $\mu = \Dirac{T_0}$ for
        some $T_0 \in \grass{\adim}{\vdim}$.
    \end{enumerate}
    We write $F \in \FAC$ if $F \in \FAC_x$ for all~$x \in \R^{\adim}$.
\end{definition}

To conclude, we introduce the following new notion of ellipticity,
named~BC. This will turn out to be equivalent to AC, see
Lemma~\ref{lem:AC-eq-BC}. Rephrasing $\FAC$ as $\FBC$ will be very useful for
the proof of Theorem \ref{roughth} and for a further understanding of
$\FAC$. Indeed, Definition \ref{def:BC} is more geometric than the algebraic
Definition \ref{def:AC}, providing a better tool to relate $\FAC$ with the other
notions of ellipticity.

\begin{definition}
    \label{def:BC}
    Let $x \in \R^\adim$. We define $\FBC_x$ to be the class of all~$\cnt^1$
    integrands~$F$ such that for any~$T \in \grass{\adim}{\vdim}$ and any Radon
    probability measure~$\mu$ over $\grass{\adim}{\vdim}$, setting $W =
    (\HM^\vdim \restrict T) \times \mu \in \Var{\vdim}(\R^{\adim})$, there holds
    \begin{displaymath}
        \delta_{F^x} W = 0
        \quad \implies \quad
        \mu = \Dirac{T} \,.
    \end{displaymath}
    We write $F \in \FBC$ if $F \in \FBC_x$ for all~$x \in \R^{\adim}$.
\end{definition}

\section{Rectifiability of test pairs}
\label{sec:rect-test-pairs}

Let $x \in \R^\adim$, $\mathcal P_1$ be the set of \emph{all} test pairs, and
$\mathcal P_2$ be the set of \emph{rectifiable} test pairs. Here we prove (see
Corollary~\ref{cor:mae:equiv}) that $\FAE_x(\mathcal P_1) = \FAE_x(\mathcal
P_2)$ and $\FAUE_x(\mathcal P_1) = \FAUE_x(\mathcal P_2)$, i.e., that the
original Almgren's definition of ellipticity~\cite[IV.1(7)]{Alm76} coincides
with the definition used in~\cite[3.16]{FangKol2017}. To~this end we need to
show an improved version of the deformation theorem, see~\ref{thm:mae:dt}. In~contrast to similar
theorems of Federer and Fleming~\cite[4.2.6-9]{Fed69}, David and
Semmes~\cite[Theorem 3.1]{DS2000}, or Fang and
Kolasiński~\cite[7.13]{FangKol2017}, this one has the special feature of
preserving the unrectifiability of the purely unrectifiable part of the deformed set.

First, we introduce some notation (modelled on~\cite{Alm1986}) needed to deal
with cubes and cubical complexes.

\begin{definition}
    \label{def:cube}
    Let $k \in \{0,1,\ldots,n\}$ and $Q = [0,1]^k \subseteq \R^k$. We say that
    $R \subseteq \R^\adim$ is a \emph{cube} if there exist $p \in \orthproj nk$,
    $o \in \R^\adim$ and $l \in (0,\infty)$ such that $R = \trans{o} \circ p^* \circ
    \scale{l} \lIm Q \rIm$. We call $\vertex R = o$ the \emph{corner} of~$R$ and
    $\side R = l$ the~\emph{side-length} of~$R$. We~also set
    \begin{itemize}
    \item $\dim(R) = k$ -- the \emph{dimension} of $R$,
    \item $\centre R = \vertex R + \frac 12 \side R (1,1,\ldots,1)$ -- the \emph{centre} of~$R$,
    \item $\cBdry R = \trans{\vertex R} \circ p^* \circ \scale{\side R} \lIm \Bdry Q \rIm$
        -- the \emph{boundary} of~$R$,
    \item $\cInt R = R \without \cBdry R$ -- the \emph{interior} of~$R$.
    \end{itemize}
\end{definition}

\begin{definition}
    \label{def:dyadic-cubes}
    Let $k \in \{0,1,\ldots,n\}$, $N \in \integers$, $Q = [0,1]^k
    \subseteq \R^k$, $e_1$, \ldots, $e_n$ be the standard basis
    of~$\R^\adim$, and $f_1$, \ldots, $f_k$ be the standard basis of~$\R^k$.

    We~define $\cubes_k^{\adim}(N)$ to be the set of all cubes $R \subseteq
    \R^{\adim}$ of the form $R = \trans{v} \circ p^* \circ \scale{2^{-N}} \lIm Q
    \rIm$, where $v \in \scale{2^{-N}}\lIm \integers^n \rIm$ and $p \in
    \orthproj nk$ is such that $p^*(f_i) \in \{ e_1,\ldots,e_n\}$ for $i =
    1,2,\ldots,k$.

    We also set
    \begin{gather*}
        \cubes_k^{\adim} = \tbcup\bigl\{ \cubes_k^{\adim}(N) : N \in \integers \bigr\} \,,
        \quad
        \cubes^{\adim} = \cubes_{\adim}^{\adim} \,,
        \quad
        \cubes_*^{\adim} = \tbcup\bigl\{ \cubes_k^{\adim} : k \in \{0,1,\ldots,\adim\} \bigr\} \,.
    \end{gather*}
\end{definition}

\begin{definition}
    \label{def:face}
    Let $k \in \{0,1,\ldots,n\}$, $N \in \integers$, and $K \in
    \cubes_k^{\adim}(N)$.  We~say that $L \in \cubes_*^{\adim}$ is a~\emph{face}
    of~$K$ if and only if $L \subseteq K$ and $L \in \cubes_j^{\adim}(N)$ for
    some $j \in \{0,1,\ldots,k\}$.
\end{definition}

\begin{definition}[\protect{cf.~\cite[1.5]{Alm1986}}]
    \label{def:admissible}
    A family of top-dimensional cubes $\mathcal F \subseteq \cubes^{\adim}$ is
    said to be \emph{admissible} if
    \begin{enumerate}
    \item $K,L \in \mathcal F$ and $K \ne L$ implies $\cInt K \cap \cInt L = \varnothing$,
    \item $K,L \in \mathcal F$ and $K \cap L \ne \varnothing$ implies $\frac 12 \le \side L / \side K \le 2$,
    \item $K \in \mathcal F$ implies $\cBdry K \subseteq \bigcup \{ L \in \mathcal F : L \ne K \}$.
    \end{enumerate}
\end{definition}

\begin{definition}[\protect{cf.~\cite[1.8]{Alm1986}}]
    \label{def:cube-complex}
    Let $\mathcal F \subseteq \cubes^{\adim}$ be admissible. We define the
    \emph{cubical complex $\CX(\mathcal F)$ of $\mathcal F$} to consist of all
    those cubes $K \in \cubes_*^{\adim}$ for which
    \begin{itemize}
    \item $K$ is a face of some cube in $\mathcal F$,
    \item if $\dim(K) > 0$, then $\side K \le \side L$ whenever $L$ is a face of
        some cube in $\mathcal F$ with $\dim(K) = \dim(L)$ and $\cInt K \cap
        \cInt L \ne \varnothing$.
    \end{itemize}
\end{definition}

\begin{definition}
    Let $k \in \nat$, $Q \subseteq \R^k$ be closed convex with non-empty
    interior, and $a \in \Int Q$. We define the \emph{central projection from
      $a$ onto $\Bdry Q$} to be the locally Lipschitz map $\pi_{Q,a} : \R^k
    \without \{a\} \to \R^k$ characterised by
    \begin{gather*}
        \pi_{Q,a}(x) \in \Bdry Q 
        \quad \text{and} \quad 
        \frac{\pi_{Q,a}(x) - a}{|\pi_{Q,a}(x) - a|} = \frac{x - a}{|x - a|} \quad \text{for $x \in \Int Q \without \{a\}$} \,,
        \\
        \pi_{Q,a}(x) = x \quad \text{for $x \in \R^k \without \Int Q$} \,.
    \end{gather*}
\end{definition}

The following lemma is a counterpart of~\cite[4.2.7]{Fed69}.

\begin{lemma}
    \label{lem:mae:good-proj}
    Assume
    \begin{gather*}
        k,N \in \nat \,,
        \quad
        \vdim < k \le \adim \,,
        \quad 
        \text{$Q \subseteq \R^{\adim}$ is a cube} \,,
        \\
        p \in \orthproj{\adim}{k} \,,
        \quad
        \im p^* = \Tan(Q,\centre{Q}) \,,
        \\
        \text{$\mu_1, \ldots, \mu_N$ are Radon measures over~$\R^{\adim}$} \,,
        \quad
        \Sigma = Q \cap {\textstyle \tbcup_{i=1}^N \spt \mu_i} \,,
        \quad
        \HM^{\vdim}(\Sigma) < \infty \,.
    \end{gather*}

    There exist $\Gamma = \Gamma(\vdim,k,N)$ and $a \in Q$ such that
    \begin{gather*}
        \dist(a, \Sigma) > 0 \,,
        \quad
        \dist(a, \cBdry{Q}) > \tfrac 14 \side Q \,,
        \\
        \text{and} \quad
        \int_Q \|\uD (\pi_{Q,a} \circ p) \|^{\vdim} \ud \mu_i \le \Gamma \mu_i(Q)
        \quad  \forall i \in \{1,\ldots,N\} \,.
    \end{gather*}
       Moreover, if $A \subseteq \Sigma$ is purely $(\HM^{\vdim},\vdim)$~unrectifiable, then $p^* \circ \pi_{Q,a} \circ p \lIm A \rIm$ is purely $(\HM^{\vdim},\vdim)$~unrectifiable.  
\end{lemma}

\begin{proof}
    Without loss of generality we shall assume $\adim = k$. Recall Definition~\ref{def:R-U}
    and Remark~\ref{rem:R-U} and let $E = \mathcal U(\Sigma)$.
    Employing~\cite[Lemma~6]{Feu2009} with $\delta$, $E$, $d$, $k$ replaced by
    $Q$, $E$, $\vdim$, $k$, we see that $\HM^{k}(B) = 0$, where
    \begin{equation}
        B = \bigl\{ a \in Q : \text{$\pi_{Q,a}\lIm E \rIm$ is not purely $(\HM^{\vdim},\vdim)$~unrectifiable} \bigr\} \,.
    \end{equation}
    Set $Q_0 = \{ x \in Q : \dist(x, \cBdry{Q}) > \frac 14 \side Q \}$.
    From~\cite[6.4]{FangKol2017} we deduce that there exists $\Gamma_0 =
    \Gamma_0(k) > 1$ such that
    \begin{displaymath}
        \| \uD \pi_{Q,a}(x) \| \le \Gamma_0 |x-a|^{-1}
        \quad \text{for all $a \in Q_0$ and all $x \in \R^k \without \{a\}$} \,.
    \end{displaymath}
    Since $\vdim < k$, there exists $\Delta = \Delta(\vdim,k) \in (0,\infty)$
    such that for all $a \in \Int Q$ there holds $\int_Q |x-a|^{-d} \ud \HM^k(a)
    < \Delta$. Using the Fubini theorem~\cite[2.6.2]{Fed69} and arguing as
    in~\cite[7.10]{FangKol2017} or in~\cite[4.2.7]{Fed69}, we find out that
    there exists $\Gamma_1 = \Gamma_1(\vdim,k,N)$ such that $\HM^{k}(A) > 0$,
    where
    \begin{displaymath}
        A = \left\{
            a \in Q_0 :
            \int_{Q} |x-a|^{-\vdim} \ud \mu_i(x)
            \le \Gamma_1 \mu_i(Q)
            \quad \text{for $i \in \{1,2,\ldots,N\}$}
        \right\} \,.
    \end{displaymath}
    We have $\HM^{k}(\Sigma) = 0$ so $\HM^k(A \without \Sigma) > 0$. Hence,
    there exists $a \in A \without (B \cup \Sigma)$ with all the desired
    properties.
\end{proof}

\begin{theorem}
    \label{thm:mae:dt}
    Assume
    \begin{gather*}
        \text{$\mathcal F \subseteq \cubes^{\adim}$ is admissible} \,,
        \quad
        \text{$\mathcal A \subseteq \mathcal F$ is finite} \,,
        \quad
        \text{$S \subseteq \R^\adim$ is a $\vdim$-set} \,,
        \\
        I = [0,1] \,,
        \quad
        J = [0,2] \,,
        \quad
        G = \Int \tbcup \mathcal A \,,
        \\
        \HM^{\vdim}(\tbcup \mathcal A \cap \Clos S) < \infty \,,
        \quad
        \text{$R = \mathcal R(S)$} \,,
        \quad
        \text{$U = \mathcal U(S)$} \,.
    \end{gather*}
    There exist $\Gamma = \Gamma(\adim,\vdim) \in (1,\infty)$, a~Lipschitz map
    $f : J \times \R^{\adim} \to \R^{\adim}$, a~finite set $\mathcal B \subseteq
    \CX(\mathcal F) \cap \cubes_{\vdim}^{\adim}$, and an open set $V \subseteq
    \R^{\adim}$ such that
    \begin{gather}
        f(0,x) = x \quad \text{for $x \in \R^{\adim}$} \,,
        \\
        f(t,x) = x \quad \text{for $(t,x) \in 
          \bigl( J \times (\R^{\adim} \without G) \cup \tbcup \mathcal B \bigr) 
          \cup \bigl( I \times \tbcup (\CX(\mathcal F) \cap \cubes_{\vdim}^{\adim}) \bigr)$} \,,
        \\
        S \subseteq V \,,
        \quad
        f \lIm J \times Q \rIm \subseteq Q \quad \text{for $Q \in \mathcal A$} \,,
        \quad
        f\lIm \{1\} \times V \rIm \cap G \subseteq \tbcup \bigl( \CX(\mathcal F) \cap \cubes_{\vdim}^{\adim} \bigr) \,,
        \\
        f\lIm \{2\} \times V \rIm \cap G = \tbcup \mathcal B \cap G \,,
        \quad
        f\lIm I \times (V \cap G) \rIm \subseteq \tbcup \mathcal A \,,
        \\
        \HM^{\vdim}(f(1,\cdot) \lIm R \cap G \rIm) \le \Gamma \HM^{\vdim}(R \cap G) \,,
        \quad
        \HM^{\vdim}(f(1,\cdot) \lIm U \cap G \rIm) \le \Gamma \HM^{\vdim}(U \cap G) \,,
        \\
        \HM^{\vdim}(f(1,\cdot) \lIm U \rIm \cap G) = 0 \,,
        \quad
        \text{$f(1,\cdot) \lIm U \rIm$ is purely $(\HM^{\vdim},\vdim)$~unrectifiable} \,,
        \\
        f(2,\cdot) \lIm f \lIm J \times V \rIm \rIm = f \lIm \{2\} \times V \rIm  \,,
        \\
        f\lIm \{2\} \times V \rIm \text{ is a strong deformation retract of } f\lIm J \times V \rIm
        \,.
    \end{gather}
\end{theorem}

\begin{proof}
    For each $Q \in \CX(\mathcal F)$ we find $p_Q \in \orthproj{\adim}{\dim Q}$
    such that $Q \subseteq \centre{Q} + \im p_Q^*$. For $k \in
    \{0,1,2,\ldots,\adim\}$ set
    \begin{displaymath}
        \mathcal A_k = \bigl\{ Q \in \CX(\mathcal F) \cap \cubes_k^{\adim} : Q \cap G \ne \varnothing \bigr\} \,.
    \end{displaymath}
    We shall perform a central projection inside the cubes of~$\mathcal A_k$ for
    $k = \adim, \adim-1, \ldots, \vdim+1$. Note that $\Bdry G \cap \bigcup
    \mathcal A_k \ne \Bdry G$ for $k < \adim$. In~fact, all the projections
    shall equal identity on~$\Bdry G$.

    Let us set
    \begin{gather}
        \mu_{1,\adim} = \HM^{\vdim} \restrict (R \cap G) \,,
        \quad
        \mu_{2,\adim} = \HM^{\vdim} \restrict (U \cap G) \,,
        \quad
        \mu_{3,\adim} = \HM^{\vdim} \restrict (S \cap G) \,,
        \\
        \varphi_{\adim}(x) = \psi_{\adim}(t,x) = x \quad \text{for $(t,x) \in I \times \R^{\adim}$} \,,
        \quad
        \delta_{\adim+1} = 1 \,,
        \\
        E = \R^{\adim} \without G \,,
        \quad
        Z_{\adim+1} = \R^{\adim} \,.
    \end{gather}
    For $k \in \{\adim-1, \adim-2, \ldots, \vdim \}$ and $i \in \{1,2,3\}$ we
    shall define Lipschitz maps $\psi_k : I \times \R^{\adim} \to \R^{\adim}$
    and $\varphi_k : \R^{\adim} \to \R^{\adim}$, Radon measures $\mu_{i,k}$
    over~$\R^{\adim}$, sets $Z_{k+1} \subseteq \tbcup \mathcal A_{k+1} \cup E$,
    and numbers $\delta_{k+1} \in (0,1)$ satisfying
    \begin{equation}
        \label{eq:mae:dt:fk}
        \left\{
        \begin{gathered}
            \spt \mu_{i,k} = \varphi_{k}\lIm \spt \mu_{i,k+1} \rIm \subseteq E \cup \tbcup \mathcal A_{k} \,,
            \quad
            \psi_k \lIm I \times Z_{k+1} \rIm = Z_{k+1} \,,
            \\ 
            \bigl( \spt \mu_{i,k+1} + \oball{0}{\delta_{k+1}} \bigr) \cap \tbcup \mathcal A_{k+1} \subseteq Z_{k+1} \,,
            \quad
            \varphi_k = \psi_k(1,\cdot) \circ \varphi_{k+1} \,,
            \\
            \psi_k(t,x) = x \quad \text{for $(t,x) \in I \times (E \cup \tbcup \mathcal A_{k})$} \,,
            \quad
            \psi_k\lIm \{1\} \times Z_{k+1} \rIm 
            \subseteq E \cup \tbcup \mathcal A_{k} \,.
        \end{gathered}
        \right.
    \end{equation}
    We proceed inductively. Assume that for some $l \in \{ \adim-1, \ldots,
    \vdim+1\}$ we have defined $\psi_k$, $\varphi_k$, $\delta_{k+1}$, $Z_{k+1}$
    and $\mu_{i,k}$ for $k \in \{\adim, \adim-1,\ldots,l+1\}$ and $i \in
    \{1,2,3\}$. For each $Q \in \mathcal A_{l+1}$ we apply
    Lemma~\ref{lem:mae:good-proj} to find~$a_Q \in Q$ satisfying
    \begin{gather}
        \label{eq:aQ-choice}
        \dist(a_Q, \spt \mu_{3,l+1}) > 0 \,,
        \quad
        \dist(a_Q, \cBdry{Q}) > \tfrac 14 \side Q \,,
        \\
        \int_{Q} \|\uD (\pi_{Q,a_Q} \circ p_Q) \|^{\vdim} \ud \mu_{i,l+1} \le \Gamma_{\ref{lem:mae:good-proj}} \mu_{i,l+1}(Q)
        \quad \text{for $i \in \{1,2,3\}$} \,,
    \end{gather} 
    and such that if $A \subseteq \spt \mu_{3,l+1}$ is purely
    $(\HM^{\vdim},\vdim)$~unrectifiable, then $p_Q^* \circ \pi_{Q,a_Q} \circ p_Q
    \lIm A \rIm$ is also purely $(\HM^{\vdim},\vdim)$~unrectifiable.

    Let $\delta_{l+1} \in (0,1)$ be such that
    \begin{multline}
        \label{eq:def-delta}
        \dist(a_Q, \spt \mu_{3,l+1}) > 2 \delta_{l+1}
        \\ \text{and} \quad
        \dist(a_Q, \cBdry{Q}) >  2 \delta_{l+1}
        \quad \text{for all $Q \in \mathcal A_{l+1}$} \,.
    \end{multline}
    Set
    \begin{equation}
        \label{eq:def-Z}
        Z_{l+1} = E \cup \bigl(
        \tbcup \mathcal A_{l+1} \without \tbcup \bigl\{ \cball{a_Q}{\delta_{l+1}} : Q \in \mathcal A_{l+1} \bigr\}
        \bigr) \,.
    \end{equation}
    Define $\tilde \psi_l : I \times Z_{l+1} \to Z_{l+1}$ by setting for $(t,x)
    \in I \times Z_{l+1}$
    \begin{equation}
        \label{eq:def-psi}
        \tilde \psi_{l}(t,x) = 
        \left\{
            \begin{aligned}
                &(1-t)x + t p_Q^* \circ \pi_{Q,a_Q} \circ p_Q(x)
                &&\text{if $x \in \cInt{Q}$ for some $Q \in \mathcal A_{l+1}$} \,,
                \\
                &\tilde \psi_{l}(t,x) = x
                &&\text{if $x \in E \cup \tbcup \mathcal A_{l}$} \,.
            \end{aligned}
        \right.
    \end{equation}
    Since for $Q \in \mathcal A_{l+1}$ the map $p_Q^* \circ \pi_{Q,a_Q} \circ
    p_Q$ equals the identity on~$\cBdry{Q}$, is Lipschitz continuous
    on~$\R^{\adim} \without \oball{a_Q}{\delta_l}$, and~$Q$ is convex, we see
    that $\tilde \psi_l$ is well defined and Lipschitz continuous. Extend
    $\tilde \psi_l$ to a~Lipschitz map $\psi_l : I \times \R^{\adim} \to
    \R^{\adim}$ using~\cite[2.10.43]{Fed69}. Next, for $i \in \{1,2,3\}$ set
    \begin{displaymath}
        \varphi_l = \psi_l(1,\cdot) \circ \varphi_{l+1} 
        \quad \text{and} \quad
        \mu_{i,l} = (\varphi_{l})_{\#}(\|\uD \varphi_{l}\|^{\vdim} \mu_{i,\adim})  \,.
    \end{displaymath}
    Note that $\| \uD \varphi_{l}\|^{\vdim}$ is bounded and $\varphi_{l}$ is
    proper, so $\mu_{i,l}$ is a~Radon measure. Also, because we assumed $\spt
    \mu_{3,l+1} \subseteq E \cup \tbcup \mathcal A_{l+1}$, we readily verify
    that
    \begin{displaymath}
        \spt \mu_{3,l} \subseteq \varphi_{l} \lIm \Clos S \rIm \subseteq E \cup \tbcup \mathcal A_{l} \,.
    \end{displaymath}
    Hence, $\psi_l$, $\varphi_l$, $\mu_{i,l}$ for $i \in \{1,2,3\}$,
    $\delta_{l+1}$, and~$Z_{l+1}$ verify~\eqref{eq:mae:dt:fk}. This concludes the
    inductive construction.

    Define
    \begin{displaymath}
        \mathcal B = \bigl\{ Q \in \mathcal A_{\vdim} : Q \subseteq \varphi_{\vdim} \lIm S \rIm \bigr\} \,.
    \end{displaymath}
    For $Q \in \mathcal A_{\vdim} \without \mathcal B$ we choose $a_Q \in
    \cInt{Q}$ so that~\eqref{eq:aQ-choice} holds and we define $\delta_{\vdim}
    \in (0,1)$ so that~\eqref{eq:def-delta} is satisfied with $l+1 = \vdim$.
    Set
    \begin{gather}
        Z_{\vdim} = E \cup \bigl(
        \tbcup \mathcal A_{\vdim} \without \tbcup \bigl\{ \cball{a_Q}{\delta_{\vdim}} : Q \in \mathcal B \bigr\}
        \bigr) \,,
        \quad 
        \tilde \psi_{\vdim-1} : Z_{\vdim} \to Z_{\vdim} \,,
        \\
        \begin{aligned}
            \tilde \psi_{\vdim-1}(t,x) &= \tilde \psi_{l}(t,x) = x
            \quad \text{if $x \in E \cup \tbcup \mathcal B \cup \tbcup \mathcal A_{\vdim-1}$} \,,
            \\
            \tilde \psi_{\vdim-1}(t,x) &= (1-t)x + t p_Q^* \circ \pi_{Q,a_Q} \circ p_Q(x)
            \\
            &\phantom{= (1-t)x + t p_Q^* \circ}\text{if $x \in \cInt{Q}$ for some $Q \in \mathcal A_{\vdim} \without \mathcal B$} \,.
        \end{aligned}
    \end{gather}
    Extend $\tilde \psi_{\vdim-1}$ to a Lipschitz map $\psi_{\vdim-1} : I \times
    \R^{\adim} \to \R^{\adim}$. Set $\varphi_{\vdim-1} = \psi_{\vdim-1}(1,\cdot)
    \circ \varphi_{\vdim}$,
    \begin{gather}
        V_{\vdim-1} = E \cup \bigl( \tbcup \mathcal B + \oball{0}{\delta_{\vdim}} \bigr) \cap Z_{\vdim} \,,
        \\ 
        \text{and} \quad
        V_{l} = \tilde \psi_{l-1}(1,\cdot)^{-1} \lIm V_{l-1} \rIm \subseteq Z_l
        \quad \forall l \in \{ \vdim, \vdim+1, \ldots, \adim \} \,.
    \end{gather}
    Note that $V_{l}$ is relatively open in~$Z_{l}$ for $l \in \{ \adim,
    \adim-1, \ldots, \vdim\}$; in~particular, $V_{\adim}$ is open
    in~$\R^{\adim}$ and, setting $V = V_{\adim}$, we get
    \begin{displaymath}
        S \subseteq V \,,
        \quad
        \varphi_{\vdim-1} \lIm V \rIm \cap G = \tbcup \mathcal B \cap G \,.
    \end{displaymath}
    We set for $l \in \{ 1, 2, \ldots, \adim-\vdim \}$ and $(t,x) \in I \times
    \R^{\adim}$ satisfying $l-1 \le (\adim-\vdim) t < l$
    \begin{displaymath}
        f(t,x) = \psi_{\adim-l}\bigl( (\adim-\vdim)t-(l-1), \varphi_{\adim-l+1}(x) \bigr) 
    \end{displaymath}
    and for $(t,x) \in [1,2] \times \R^{\adim}$
    \begin{displaymath}
        f(t,x) = \psi_{\vdim-1}\bigl( t-1, \varphi_{\vdim}(x) \bigr)  \,.
    \end{displaymath}
    This defines a Lipschitz map $f : J \times \R^{\adim} \to \R^{\adim}$. From
    the construction it follows that $f\lIm \{1\} \times U \rIm$ is purely
    $(\HM^{\vdim},\vdim)$~unrectifiable and $f(1,\cdot)\lIm U \rIm \cap G
    \subseteq \tbcup (\CX(\mathcal F) \cap \cubes_{\vdim}^{\adim})$, so
    \begin{displaymath}
        \HM^{\vdim}(f(1,\cdot)\lIm U \rIm \cap G) = 0 \,.
    \end{displaymath}

    Now, we need to verify the required estimates. For brevity of the notation
    let us set
    \begin{displaymath}
        g = f(1,\cdot) \quad \text{and} \quad
        \eta_k = \psi_k(1,\cdot) \quad \text{for $k \in \{\vdim, \vdim+1 \ldots, \adim\}$} \,.
    \end{displaymath}
    Observe that if $Q \in \mathcal F$, then $\HM^{0}(\{ R \in \mathcal F : R
    \cap Q \ne \varnothing \}) \le 4^{\adim}$. Note also that for $k \in \{
    \vdim, \vdim+1, \ldots, \adim-1 \}$ and $i \in \{1,2,3\}$ we have
    \begin{multline*}
        (\varphi_{k+1})_{\#}\bigl( \| \uD \varphi_{k+1} \|^{\vdim} \mu_{i,\adim}
        \restrict \varphi_{k}^{-1} \lIm \tbcup \mathcal A_k \rIm \bigr)
        \\
        = (\varphi_{k+1})_{\#}\bigl( \| \uD \varphi_{k+1} \|^{\vdim} \mu_{i,\adim} \bigr)
        \restrict \varphi_{k+1} \lIm \varphi_{k}^{-1} \lIm \tbcup \mathcal A_k \rIm \rIm
        \\
        = \mu_{i,k+1} \restrict \eta_k^{-1} \lIm \tbcup \mathcal A_k \rIm
        \le \mu_{i,k+1} \restrict \tbcup \mathcal A_{k+1} \,,
    \end{multline*}
    so we obtain
    \begin{multline}
        \label{eq:mae:dt:iter}
        \mu_{i,k}(\tbcup \mathcal A_k) 
        = \int_{\varphi_k^{-1} \left \lIm \tbcup \mathcal A_k \right\rIm} \| \uD \varphi_k \|^{\vdim} \ud \mu_{i,\adim}
        \\
        \le \int_{\varphi_k^{-1} \left \lIm \tbcup \mathcal A_k \right\rIm} \| \uD \eta_k \circ \varphi_{k+1} \|^{\vdim} \| \uD \varphi_{k+1} \|^d \ud \mu_{i,\adim}
        = \int_{\eta_k^{-1} \left \lIm \tbcup \mathcal A_k \right\rIm} \| \uD \eta_k \|^{\vdim} \ud \mu_{i,k+1}
        \\
        \le \int_{\tbcup \mathcal A_{k+1}} \| \uD \eta_k \|^{\vdim} \ud \mu_{i,k+1}
        \le \sum_{Q \in \mathcal A_{k+1}} \int_{Q} \| \uD \eta_k \|^{\vdim} \ud \mu_{i,k+1}
        \\
        = \sum_{Q \in \mathcal A_{k+1}} \int_{Q} \| \uD (\pi_{Q,a_Q} \circ p_Q) \|^{\vdim} \ud \mu_{i,k+1}
        \\
        \le \Gamma_{\ref{lem:mae:good-proj}} \sum_{Q \in \mathcal A_{k+1}} \mu_{i,k+1}(Q)
        \le 4^{\adim} \Gamma_{\ref{lem:mae:good-proj}} \mu_{i,k+1}(\tbcup \mathcal A_{k+1}) \,.
    \end{multline}
    In~particular, setting $\Sigma_1 = R \cap G$, $\Sigma_2 = U \cap G$ and
    employing~\cite[7.12]{FangKol2017} we obtain for $i \in \{1,2\}$
    \begin{multline*}
        \HM^{\vdim}(g \lIm \Sigma_i \rIm \cap \tbcup \mathcal A_{\vdim})
        = \HM^{\vdim}(\varphi_d \lIm \Sigma_i \rIm \cap \tbcup \mathcal A_{\vdim})
        \\
        \le \int_{\varphi_{\vdim}^{-1} \left\lIm \tbcup \mathcal A_{\vdim} \right\rIm} \| \uD \varphi_{\vdim} \|^{\vdim} \ud \mu_{i,\adim}
        = \mu_{i,\vdim}(\tbcup \mathcal A_{\vdim})
        \\
        \le \bigl( 4^{\adim} \Gamma_{\ref{lem:mae:good-proj}} \bigr)^{\adim - \vdim} \mu_{i,\adim}(\tbcup \mathcal A_{\adim})
        = \bigl( 4^{\adim} \Gamma_{\ref{lem:mae:good-proj}} \bigr)^{\adim - \vdim} \HM^{\vdim}(\Sigma_i) \,.
    \end{multline*}
    Estimating as in~\eqref{eq:mae:dt:iter}, we also get
    \begin{multline*}
        \HM^{\vdim}(g \lIm \Sigma_i \rIm \without \tbcup \mathcal A_{\vdim})
        = \HM^{\vdim}(\varphi_d \lIm \Sigma_i \rIm \without \tbcup \mathcal A_{\vdim})
        \le \int_{G \cap \varphi_{\vdim}^{-1} \left\lIm \Bdry G \right\rIm} \| \uD \varphi_{\vdim} \|^{\vdim} \ud \mu_{i,\adim}
        \\
        \le \int_{\varphi_{\vdim+1} \lIm G \rIm \cap \eta_{\vdim}^{-1} \left\lIm \Bdry G \right\rIm} \| \uD \eta_{\vdim} \|^{\vdim} \ud \mu_{i,\vdim+1}
        \le \int_{\tbcup \mathcal A_{\vdim+1}} \| \uD \eta_{\vdim} \|^{\vdim} \ud \mu_{i,\vdim+1}
        \\
        \le 4^{\adim} \Gamma_{\ref{lem:mae:good-proj}} \mu_{i,\vdim+1}(\tbcup \mathcal A_{\vdim+1})
        \le (4^{\adim} \Gamma_{\ref{lem:mae:good-proj}})^{\adim - \vdim} \HM^{\vdim}(\Sigma_i)  \,.
    \end{multline*}
    This gives the desired estimates.
\end{proof}

\begin{remark}
    \label{rem:image-warning}
    Observe that 
    \begin{equation*}
        f(1,\cdot) \lIm S \rIm \cap G \subseteq \tbcup \bigl( \CX(\mathcal F) \cap \cubes_{\vdim}^{\adim} \bigr) \qquad 
        \text{but} \qquad
        f(1,\cdot) \lIm S \cap G \rIm \subseteq \tbcup \bigl( \CX(\mathcal F) \cap \cubes_{\vdim}^{\adim} \bigr) \cup \Bdry G \,.
    \end{equation*}
\end{remark}

\begin{remark}
    \label{rem:dt-retract}
    Define
    \begin{gather*}
        \tilde Q = \tbcup \{ R \in \mathcal F : R \cap Q \ne \varnothing \} 
        \quad \forall Q \in \mathcal F \,,
        \\
        H = \tbcup \{ Q \in \mathcal A : \tilde Q \subseteq \tbcup \mathcal A \} \,,
        \quad \text{and} \quad
        W = V \cap G \,.
    \end{gather*}
    Assume that $S$ is separated from $E = \R^{\adim} \without G$ in the sense
    that $S \subseteq H$. Then $W$ is an open neighborhood of~$S$
    in~$\R^{\adim}$ with
    \begin{displaymath}
        f\lIm J \times S \rIm \subseteq f\lIm J \times W \rIm \subseteq W 
    \end{displaymath}
    and $f(2,\cdot) \lIm W \rIm = \tbcup B$ is a strong deformation retract
    of~$S$.
\end{remark}

In the next lemma given a test pair~$(S,D)$ we construct a Lipschitz deformation
$f : \R^{\adim} \to \R^{\adim}$ which modifies the rectifiable part~$R$ of~$S$
only on a set of small measure and transforms the unrectifiable part~$I$ into
a~nullset. The construction works as follows. The set~$R$ can be represented,
up~to a~set of arbitrarily small measure, as a~finite disjointed collection $\{
F_1, \ldots, F_N\}$, where each $F_i$ is a~compact subsets of the~graph of
a~$\cnt^1$ map~$\psi_i : T_i \to T_i^{\perp}$ for some $T_i \in
\grass{\adim}{\vdim}$. Since the pieces $F_i$ are compact and pairwise disjoint,
there is a~positive distance $70\delta$ between them. To deal with the part
of~$I$ which lies at least $4\delta$ away from~$F = \tbcup_{i=1}^N F_i$ we
employ the deformation theorem~\ref{thm:mae:dt} and obtain the map $g : \R^\adim
\to \R^\adim$ which does not move points of~$F$, converts the pary of~$I$ away
from~$F$ into a nullset, and preserves unrectifiability of the part of~$I$ close
to~$F$. After this step the unrectifiable part of~$g \lIm S \rIm$ lies entirely
in $4\delta$-neighbourhood of~$F$. Next, for each~$i$ we employ the
Besicovitch-Federer projection theorem to find $P_i \in \grass{\adim}{\vdim}$
such that the associated orthogonal projection~$\project{P_i}$ kills the measure
of the unrectifiable part of~$g\lIm S \rIm$. We replace $\psi_i$ with $\varphi_i
: P_i \to P_i^{\perp}$ so that the graphs of~$\psi_i$ and~$\varphi_i$ coincide
and we define a~projection $\pi_i = \project{P_i} + \varphi_i \circ
\project{P_i}$ onto the graph of $\psi_i$. The map $\pi_i$ does not move points
of~$F_i$ and carries the unrectifiable part of~$g\lIm S \rIm$ into a
nullset. The final step is to combine all the maps~$\pi_i$ into a single map~$h$
using simple interpolation, which is possible since $F_i$ is at least $70\delta$
away from $F_j$ if $i \ne j$. The final deformation is $f = h \circ g$. There is
still a~small problem with $f$: we do not know how~$f$ acts on the boundary~$B$
of~$D$ and we want $(f\lIm S \rIm,D)$ to be a~test pair. To deal with that we
artificially introduce the set $F_0 = T \cap (B + \cball 0{\delta})$ and the
map $\psi_0 : T \to T^{\perp}$, where $T \in \grass{\adim}{\vdim}$ is such that
$D \subseteq T$. After that, the whole construction yields a correct map.

\begin{lemma}
    \label{lem:mae:rect-test-pair}
    Assume
    \begin{gather}
        \text{$(S,D)$ is a test pair} \,,
        \quad
        T = \Tan(D,0) \,,
        \quad
        B = T \cap \Bdry{\cball 01} \,,
        \\
        R = \mathcal R(S) \,,
        \quad
        I = \mathcal U(S) \,.
    \end{gather}
    For each $\varepsilon \in (0,1)$ there exists a Lipschitz map $f :
    \R^{\adim} \to \R^{\adim}$ such that
    \begin{gather*}
        f(x) = x \quad \text{for $x \in B$} \,,
        \quad
        \HM^{\vdim}(f \lIm I \rIm) = 0 \,,
        \quad
        \HM^{\vdim}\bigl( (R \without f \lIm R \rIm) \cup (f \lIm R \rIm \without R) \bigr) \le \varepsilon \,.
    \end{gather*}
    In~particular, $f \lIm S \rIm$ is $(\HM^{\vdim},\vdim)$~rectifiable and $(f
    \lIm S \rIm,D)$ is a~rectifiable test pair.
\end{lemma}

\begin{proof}
    We define
    \begin{equation}
        \label{eq:iota-def}
        \iota = \bigl( 2 + 45 \Gamma_{\ref{thm:mae:dt}} + 45 \bigr)^{-1} \varepsilon \,.
    \end{equation}
    Since $\HM^{\vdim}(B) = 0$ we can find $\delta_0 \in (0, \frac 14)$ such that
    \begin{displaymath}
        \HM^{\vdim}\bigl( (B + \cball 0{\delta_0}) \cap S \bigr) < \iota \,.
    \end{displaymath}
    Employing~\cite[3.2.29, 3.1.19(5), 2.8.18, 2.2.5]{Fed69} we find $Z
    \subseteq \R^\adim$ and for each $i \in \nat$ a~vectorspace $T_i \in
    \grass{\adim}{\vdim}$, a~compact set $K_i \subseteq T_i$, and a~$\cnt^1$ map
    $\psi_i : T_i \to T_i^{\perp}$ such that, denoting $\bar F_i = \{ x +
    \psi_i(x) : x \in K_i \}$, it holds
    \begin{equation}\label{11}
        \bar F_i \cap \bar F_j = \varnothing \quad \forall i \ne j \,, \quad 
        R = Z \cup {\textstyle \tbcup_{i=1}^{\infty} \bar F_i } \,,
        \quad
        \HM^{\vdim}(Z) = 0 \,,
        \quad
        \Lip \psi_i \le 1 \,.
    \end{equation}
    Since $\HM^{\vdim}(R) < \infty$ we can find $N \in \nat$ such that
    \begin{equation*}
        \HM^{\vdim}(R \without {\textstyle \tbcup_{i=1}^{N} \bar F_i}) < \iota \,.
    \end{equation*}
    Set
    \begin{equation}
        \delta = 80^{-1}\min \bigl\{ \delta_0,
        \inf \bigl\{ |x-y| : x \in \bar F_i ,\, y \in \bar F_j ,\,i,j \in \{1,\ldots,N\} \,, i \ne j \bigr\} \bigr\}
        < 80^{-1}\,.
    \end{equation}
    Note that $\delta > 0$ because the sets $\bar F_i$ are mutually disjoint and
    compact. Define
    \begin{displaymath}
        F_0 = T \cap (B + \cball 0{\delta}) \,,
        \quad
        T_0 = T \,,
        \quad \text{and} \quad
        \psi_0 : T \to T^{\perp} \text{ by } \psi_0(x) = 0 \text{ for } x \in T \,.
    \end{displaymath}
    For $i \in \{1,\ldots,N\}$ set
    \begin{displaymath}
        F_i = \bar F_i \without ( F_0 + \oball 0{70\delta} )
        \quad \text{and} \quad
        F = {\textstyle \tbcup_{i=0}^{N} F_i} \,.
    \end{displaymath}
    Clearly we have
    \begin{gather}
        \label{eq:R-F-meas}
        B \subseteq F \,,
        \quad
        \HM^{\vdim}(R \without F) \le
        \HM^{\vdim}(R \without {\textstyle \tbcup_{i=1}^{N} \bar F_i})
        + \HM^{\vdim}\bigl( (B + \cball 0{\delta_0}) \cap S \bigr)
        < 2 \iota
        \\
        \label{22}
        \text{and} \quad
        |x-y| \ge 70 \delta \quad \text{whenever $x \in F_i$, $y \in F_j$, $i,j \in \{0,1,\ldots,N\}$, $i \ne j$} \,.
    \end{gather}
    Let $L \in \nat$ be such that $2^{-L} < \delta n^{-1/2} \le 2^{-L+1}$ so
    that $\diam Q < \delta$ whenever $Q \in \cubes_{\adim}^{\adim}(L)$.
    We~define
    \begin{gather}
        \mathcal F = \cubes_{\adim}^{\adim}(L) \,,
        \quad
        \widetilde Q = \tbcup \{ Q' \in \mathcal F : Q' \cap Q \ne \varnothing \}
        \quad \text{for every $Q \in \mathcal F$} \,,
        \\
        \mathcal A = \bigl\{ Q \in \mathcal F :
        \widetilde Q \cap I \ne \varnothing ,\,
        Q \cap (F + \cball 0{2\delta}) = \varnothing \bigr\}  \,,
        \quad
        G = \Int \tbcup \mathcal A
        \,.
    \end{gather}
    Observe that
    \begin{equation}
        \label{eq:dist-4d}
        \bigl\{ x \in I : \dist(x,F) \ge 4 \delta \bigr\}
        \subseteq \tbcup \bigl\{ Q \in \mathcal A : \widetilde Q \subseteq G \bigr\}
        \subseteq G \,.
    \end{equation}
    We apply Theorem~\ref{thm:mae:dt} to obtain a~Lipschitz continuous map $g :
    \R^{\adim} \to \R^{\adim}$ such that
    \begin{gather}
        g(x) = x \quad \text{for $x \in \R^{\adim} \without G$} \,,
        \quad
        \text{$g\lIm I \rIm$ is purely $(\HM^{\vdim},\vdim)$~unrectifiable} \,,
        \\
        \label{eq:im-reg-est}
        \HM^{\vdim}(g \lIm R \cap G \rIm) 
        \le \Gamma_{\ref{thm:mae:dt}} \HM^{\vdim}(R \cap G)
        \le \Gamma_{\ref{thm:mae:dt}} \HM^{\vdim}(R \without F)
        \le  \iota  \cdot \Gamma_{\ref{thm:mae:dt}} \,,
        \\
        \label{eq:dt-zero-im}
        \HM^{\vdim}(g \lIm I  \rIm \cap G) = 0 \,,
        \quad
        \HM^{\vdim}(g \lIm I \cap G \rIm) 
        \le \Gamma_{\ref{thm:mae:dt}} \HM^{\vdim}(I \cap G) < \infty \,.
    \end{gather}
    In~particular, from~\eqref{eq:dt-zero-im}, \eqref{eq:dist-4d}, and the fact
    that $g\lIm Q \rIm \subseteq Q$ for all $Q \in \mathcal F$ we deduce
    \begin{equation}
        \label{eq:far-I-zero}
        \HM^\vdim\bigl( g \bigl\lIm \bigl\{ x \in I : \dist(x,F) \ge 4 \delta \bigr\} \bigr\rIm \bigr) = 0 \,.
    \end{equation}
    For each $i \in \{ 0, 1, \ldots, N \}$, we employ the Besicovitch-Federer
    projection theorem~\cite[3.3.15]{Fed69} to choose $P_i \in
    \grass{\adim}{\vdim}$ such that
    \begin{equation}\label{33}
        \|\project{P_i} - \project{T_i}\| < 1/100
        \quad \text{and} \quad
        \HM^{\vdim}(\project{P_i} \circ g \lIm I \rIm) = 0 \,.
    \end{equation}
    Thanks to \eqref{11} and \eqref{33}, we can
    apply~\cite[Lemma~3.2]{KSvdM2015} to conclude that for every $i \in
    \{0,1,\ldots,N\}$ there exists a~$\cnt^1$~function $\varphi_i : P_i \to
    P_i^{\perp}$ such that $\{ x + \psi_i(x) : x \in T_i \} = \{ x +
    \varphi_i(x) : x \in P_i \}$ and $\Lip \varphi_i \le 2$. Next, for every $i
    \in \{0,1,\ldots,N\}$ we define the projection onto the graph of $\varphi_i$
    by the formula
    \begin{displaymath}
        \pi_i : \R^{\adim} \to \R^{\adim} \,,
        \quad
        \pi_i(x) = \project{P_i}x + \varphi_i(\project{P_i}x) \quad \text{for $x \in \R^{\adim}$} \,.
    \end{displaymath}
    Note that $\Lip \pi_i \le 1 + \Lip \varphi_i \le 3$. We choose a~smooth map
    $\gamma : \R \to \R$ such that
    \begin{displaymath}
        \gamma(t) = 0 \quad \text{for $t > 10 \delta$} \,,
        \quad
        \gamma(t) = 1 \quad \text{for $t < 5 \delta$} \,,
        \quad
        -\frac{1}{\delta} \le \gamma'(t) \le 0
    \end{displaymath}
    and we define $\cnt^{\infty}$ maps $f, h, \lambda_0, \lambda_1, \ldots,
    \lambda_N : \R^{\adim} \to \R^{\adim}$ by
    \begin{gather}
        \lambda_i(x) = \gamma(\dist(x,F_i)) \pi_i(x) + (1 - \gamma(\dist(x,F_i))) x \quad \text{for $i \in \{0,1,2,\ldots,N\}$}\,,
        \\
        h = \lambda_0 \circ \lambda_1 \circ \cdots \circ \lambda_N \,,
        \quad
        f = h \circ g \,.
    \end{gather}
    We remark that for every $x \in \R^{\adim}$, if there exists $i \in \{
    0,1,\ldots, N \}$ and $y \in F_i$ satisfying $|x-y| = \dist(x,F_i) \le 10
    \delta$, then $\pi_i(y) = y$ and
    \begin{equation}
        \label{44}
        |x - \pi_i(x)|
        \le |x-y| + |\pi_i(y) - \pi_i(x)| + |y - \pi_i(y)|
        \le 10 \delta + 3\cdot 10 \delta \le 40 \delta \,.
    \end{equation}
    In particular, \eqref{44} implies that 
    \begin{equation}
        \label{55}
        \dist(\lambda_i(x),F_i)\leq \dist(\lambda_i(x),x)+\dist(x,F_i) \leq \dist(\pi_i(x),x)+10 \delta \leq 50 \delta \,,
    \end{equation}
    which in turn, combined with~\eqref{22}, implies that $h(x)=\lambda_i(x)$
    and that the index~$i$ is unique for~$x$. Moreover, since the map
    $\dist(\cdot,F_i)$ is $1$-Lipschitz, we get
    \begin{displaymath}
        \| \uD h(x) \| = \| \uD \lambda_i(x) \|
        \le \delta^{-1} |\pi_i(x) - x| + \| \uD(\pi_i - \id{\R^\adim})(x) \| + 1
        \le 45 \,.
    \end{displaymath}
    On the other hand, if $x \in \R^{\adim}$ is such that $\dist(x,F_i) > 10
    \delta$ for every $i \in \{ 1,\ldots, N \}$, then $h(x)=x$. Hence, we get
    \begin{equation}
        \label{66}
        \Lip h \le 45 \,.
    \end{equation}
    Since, by~\eqref{eq:far-I-zero}, the unrectifiable part of $g \lIm S \rIm$
    lies in $4\delta$-neighbourhood of~$F$ and for each $i \in \{0,1,\ldots,N\}$
    the maps $h$, $\lambda_i$, and $\pi_i$ are all equal in
    $5\delta$-neighbourhood of~$F_i$ we see that
    \begin{displaymath}
        \HM^{\vdim}(f \lIm I \rIm) = 0 \,.
    \end{displaymath}
    Moreover, since $f(x) = x$ for $x \in F$ we have
    \begin{gather}
        R \without f \lIm R \rIm
        \subseteq R \without f \lIm R \cap F \rIm
        = R \without (R \cap F) = R \without F \,,
        \\
        f \lIm R \rIm \without R
        \subseteq f \lIm R \rIm \without (R \cap F)
        = f \lIm R \rIm \without f \lIm R \cap F \rIm
        \subseteq f \lIm R \without (R \cap F) \rIm
        =  f \lIm R \without F \rIm \,;
    \end{gather}
    hence, recalling~\eqref{eq:R-F-meas}, \eqref{eq:im-reg-est},
    and~\eqref{eq:iota-def}, we get
    \begin{multline}
        \HM^{\vdim} \bigl( (R \without f \lIm R \rIm) \cup (f \lIm R \rIm \without R) \bigr)
        \le 
        \HM^{\vdim}(R \without F) +  \HM^{\vdim}(f \lIm R \without F \rIm)
        \\
        \le 2 \iota + \Lip h \cdot \HM^{\vdim}(g \lIm R \without F \rIm)
        \le 2 \iota + 45 \HM^{\vdim}(g \lIm R \cap G \rIm) + 45 \HM^{\vdim}(R \without (G \cup F))
        \\
        \le 2 \iota + 45  \Gamma_{\ref{thm:mae:dt}} \iota + 45 \iota
        \le \iota ( 2 + 45 \Gamma_{\ref{thm:mae:dt}} + 45 ) \le \varepsilon \,.
        \qedhere
    \end{multline}
\end{proof}

\begin{remark}
    The difficulty in proving Lemma~\ref{lem:mae:rect-test-pair} stems from the
    situation when $\HM^{\vdim}(R \cap \Clos I) > 0$;
    cf.~\cite[4.2.25]{Fed69}. In this case one cannot argue that
    \begin{displaymath}
        {\textstyle \lim_{r \downarrow 0} \HM^{\vdim}((I + \oball 0r) \cap R) = 0 }
    \end{displaymath}
    so it is not possible to separate the unrectifiable part of~$S$ from the
    rectifiable part. However, since $R$~has a nice (rectifiable) structure and
    $I$~can be easily squashed to a~set of~$\HM^{\vdim}$ measure zero by means
    of Besicovitch-Federer projection theorem~\cite[3.3.15]{Fed69}, we can find
    nice Lipschitz deformations which produce ``holes'' in~$I$ and do not move
    most of~$R$.
\end{remark}

\begin{corollary}
    \label{cor:mae:equiv}
    Let $x \in \R^\adim$, $\mathcal P_1$ be the set of all test pairs, and $\mathcal
    P_2$ be the set of rectifiable test pairs. Then
    \begin{displaymath}
        \FAE_x(\mathcal P_1) = \FAE_x(\mathcal P_2)
        \quad \text{and} \quad
        \FAUE_x(\mathcal P_1) = \FAUE_x(\mathcal P_2) \,.
    \end{displaymath}
\end{corollary}

\begin{proof}
    Since $\mathcal P_2 \subseteq \mathcal P_1$ we clearly have $\FAE_x(\mathcal
    P_1) \subseteq \FAE_x(\mathcal P_2)$ and $\FAUE_x(\mathcal P_1) \subseteq
    \FAUE_x(\mathcal P_2)$. Hence, it suffices to prove the reverse
    inclusions. Take any test pair $(S,D) \in \mathcal P_1$ and set
    \begin{displaymath}
        T = \Tan(D,0) \,,
        \quad
        B = T \cap \cball 01 \,,
        \quad
        R = \mathcal R(S) \,,
        \quad \text{and} \quad
        I = \mathcal U(S) \,.
    \end{displaymath}
    For each $k \in \nat$ apply Lemma~\ref{lem:mae:rect-test-pair} with $\varepsilon =
    1/k$ to obtain a~map~$f_k : \R^{\adim} \to \R^{\adim}$ satisfying
    \begin{gather}
        \Lip f_k < \infty \,,
        \quad
        f_k(x) = x \quad \text{for $x \in B$} \,,
        \\
        \HM^{\vdim}(f \lIm I \rIm) = 0 \,,
        \quad
        \HM^{\vdim}\bigl( (R \without f_k \lIm R \rIm) \cup (f_k \lIm R \rIm \without R) \bigr) \le \tfrac 1k \,.
    \end{gather}
    Then $(S_k,D) = (f_k \lIm S \rIm, D)$ is a rectifiable test pair for each $k
    \in \nat$, hence for any integrand $F$ we~have
    \begin{displaymath}
        \Psi_{F^x}(S_k) - \Psi_{F^x}(D) = \Phi_{F^x}(S_k) - \Phi_{F^x}(D) \,.
    \end{displaymath}
    Observe that
    \begin{displaymath}
        \bigl| \lim_{k \to \infty} \HM^{\vdim}(S_k) - \HM^{\vdim}(R) \bigr| = 0 \,;
        \quad \text{hence, also} \quad
        \bigl| \lim_{k \to \infty} \Phi_{F^x}^{\vdim}(S_k) - \Phi_{F^x}^{\vdim}(R) \bigr| = 0 \,.
    \end{displaymath}
    Thus, if $F \in \FAUE_x(\mathcal P_2)$, then
    \begin{multline*}
        \Psi_{F^x}(S) - \Psi_{F^x}(D) = \Psi_{F^x}(I) + \lim_{k \to \infty} \Phi_{F^x}(S_k) - \Phi_{F^x}(D)
        \\
        \ge \Psi_{F^x}(I) + c \bigl( \HM^{\vdim}(R) - \HM^{\vdim}(D) \bigr)
        \\
        \ge \inf \bigl( \{c\} \cup \im F^x \bigr) \bigl( \HM^{\vdim}(S) - \HM^{\vdim}(D) \bigr)
        \,.
    \end{multline*}
    Similarly, if $F \in \FAE_x(\mathcal P_2)$, then
    \begin{displaymath}
        \Psi_{F^x}(S) - \Psi_{F^x}(D) = \Psi_{F^x}(I) + \lim_{k \to \infty} \Phi_{F^x}(S_k) - \Phi_{F^x}(D) > \Psi_{F^x}(I) \ge 0 \,.
        \qedhere
    \end{displaymath}
\end{proof}

\begin{remark}
    Recalling Remark~\ref{rem:test-pairs}, from Corollary~\ref{cor:mae:equiv} we
    deduce that definitions~\cite[IV.1(7)]{Alm76} and~\cite[3.16]{FangKol2017}
    are equivalent.
\end{remark}

\section{Existence of a minimiser for an integrand in $\FwBC$}
\label{Plateau}
In this section we provide a solution to the set theoretical formulation of the
anisotropic Plateau problem under the assumption $F \in \FwBC$. Since~$\FwBC$
will be proven to be equivalent to~$\FAC$, see Lemma~\ref{lem:AC-eq-BC}, this
section reproves~\cite[Theorem 1.8]{DDG2017b} without referring to the results
of~\cite{DDG2016rect}.
\begin{definition}
    \label{def:adm-deform}
    Let $U \subseteq \R^{\adim}$ be open. We say that $f : \R^{\adim} \to \R^{\adim}$ is
    a~\emph{basic deformation in~$U$} if $f$ is of class~$\cnt^1$ and there
    exists a bounded convex open set $V \subseteq U$ such that
    \begin{displaymath}
        f(x) = x \quad \text{for every $x \in \R^{\adim} \without V$}
        \quad \text{and} \quad
        f \lIm V \rIm  \subseteq V \,.
    \end{displaymath}
    If $f \in \cnt^1(\R^{\adim},\R^{\adim})$ is a composition of a finite number of basic
    deformations, then we say that $f$ is an~\emph{admissible deformation
      in~$U$}. The set of all such deformations shall be denoted $\adm{U}$.
\end{definition}

\begin{definition}[\protect{cf.~\cite[2.10.21]{Fed69}}]
    Whenever $K \subseteq \R^{\adim}$ is compact and $A,B \subseteq \R^{\adim}$, we
    define~$\HDK{K}(A,B)$ by
    \begin{align}
      \HDK{K}(A,B)
      &= \sup \bigl\{ | \dist(x,A) - \dist(x,B)| : x \in K \bigr\}
      \\
      &= \max \bigl\{
        \sup \{ \dist(x,A) : x \in K \cap B \} \,,\,
        \sup \{ \dist(x,B) : x \in K \cap A \}
        \bigr\} 
        \,.
    \end{align}
\end{definition}

\begin{definition}
    \label{def:goodclass}
    Let $U \subseteq \R^{\adim}$ be an open set. We say that $\mathcal C$ is
    a~\emph{good class in $U$} if
    \begin{enumerate}
    \item
        \label{i:gc:nonempty}
        $\mathcal C \ne \varnothing$;
    \item
        \label{i:gc:compact}
        each $S \in \mathcal C$ is a closed subset of~$\R^{\adim}$;
    \item
        \label{i:gc:deformation}
        if $S \in \mathcal C$ and $f \in \adm{U}$, then $f \lIm S \rIm \in
        \mathcal C$;
    \end{enumerate}
\end{definition}

\begin{remark}
    Definition~\ref{def:goodclass} differs from~\cite[3.4]{FangKol2017} by not
    assuming that the class is closed under Hausdorff convergence.
\end{remark}

Combining~\cite[11.2, 11.3, 11.7, 11.8(a)]{FangKol2017} we obtain the following.

\begin{theorem}
    \label{thm:FK}
    Let $U \subset \R^\adim$ be an open set, $\mathcal C$ be a~good class in~$U$,
    and $F$ be a~bounded $\cnt^0$~integrand. Set $\mu = \inf\bigl\{ \Phi_F(T
    \cap U) : T \in \mathcal C \bigr\}$.

    If $\mu \in (0,\infty)$, then there exist $V \in \Var{\vdim}(U)$, $S
    \subseteq \R^\adim$ closed, and $\{ S_i \in \mathcal C : i \in \nat \}$ such
    that
    \begin{enumerate}
    \item
        \label{i:FK:rect}
        $S \cap U$ is $(\HM^{\vdim},\vdim)$ rectifiable. In~particular
        $\HM^{\vdim}(S \cap U) < \infty$.

    \item
        \label{i:FK:var-lim}
        $\lim_{i \to \infty} \var{\vdim}(S_i \cap U) = V$ in $\Var{m}(U)$.

    \item
        \label{i:FK:min}
        $\lim_{i \to \infty} \Phi_F(S_i \cap U) = \Phi_F(V) = \mu$.

    \item $\spt \|V\| \subseteq S \cap U$ and $\HM^{\vdim}(S \cap U \without
        \spt \|V\|) = 0$.

    \item 
        \label{i:FK:abs-cont}
        The measures $\|V\|$ and $\HM^{\vdim} \restrict S$ are mutually
        absolutely continuous.

    \item $\lim_{i \to \infty} \HDK{K}(S_i \cap U, S \cap U) = 0$ for any
        compact set $K \subseteq U$.

    \item
        \label{i:FK:HM-far}
        For any compact set $K \subseteq U$ we have
        \begin{displaymath}
            \lim_{i \to \infty} \sup \bigl\{
            r \in \R : \HM^m(\{ x \in S_i \cap K : \dist(x, \spt \|V\| \cup \R^\adim \without U) \ge r \}) > 0
            \bigr\} = 0 \,.
        \end{displaymath}
        
    \item
        \label{i:FK:unrect}
        If $\bar S_i = \mathcal U(S_i \cap U)$, then
       $$
            \lim_{r \downarrow 0} \lim_{i \to \infty} r^{-\vdim} \HM^{\vdim}(\bar S_i \cap \cball xr) = 0
            \quad \text{for $\|V\|$-a.e.~$x$} \quad
            \text{and} \quad
            \lim_{i \to \infty} \HM^{\vdim}(\bar S_i) = 0 \,.
       $$

    \item $\density^{\vdim}(\|V\|,x) \ge 1$ for $\|V\|$ almost all~$x$.

    \item 
        \label{i:FK:Tan}
        For $\HM^{\vdim}$ almost all~$x \in \spt \|V\|$ we have
        \begin{displaymath}
            \Tan^{\vdim}(\|V\|,x) = \Tan(\spt\|V\|,x) \in \grass{\adim}{\vdim} \,.
        \end{displaymath}

    \item
        \label{i:FK:compact}
        If $\R^{\adim} \without U$ is compact and there exists
        a~$\Phi_F$-minimising sequence in $\mathcal C$ consisting only of compact
        sets (but not necessarily uniformly bounded), then
        \begin{displaymath}
            \diam(\spt\|V\|) < \infty
            \quad \text{and} \quad
            \sup \bigl\{ \diam(S_i \cap U) : i \in \nat \bigr\} < \infty \,.
        \end{displaymath}
    \end{enumerate}
\end{theorem}

\begin{lemma}
    \label{lem:fst-var-mini}
    Assume $U \subseteq \R^{\adim}$ is open, $V \in \Var{\vdim}(U)$, $\mathcal
    C$ is a good class, $F$ is a bounded $\cnt^0$~integrand, $\mu = \inf\{
    \Phi_F(P) : P \in \mathcal C \}$, $\Phi_{F}(V) = \mu$, and either $V =
    \var{\vdim}(S \cap U)$ for some $(\HM^{\vdim},\vdim)$~rectifiable set $S \in
    \mathcal C$, or there exists a sequence $\{ S_i \in \mathcal C : i \in
    \nat\}$ such that
    \begin{displaymath}
        \lim_{i \to \infty} \var{\vdim}(S_i \cap U) = V 
        \quad \text{and} \quad
        \lim_{j \to \infty} \HM^{\vdim}(\mathcal U(S_j \cap U)) = 0 \,.
    \end{displaymath}
    Then $\delta_F V  = 0 \,.$
\end{lemma}

\begin{proof}
    The proof can be found, with a slightly different notation, in
      \cite[Section 5.1]{DeR2016}. For the sake of the exposition we report it
      below.
    
    Assume there exists $g \in \VF(U)$ such that $\delta_F V(g) \ne 0$. Since
    $\spt g$ is compact, using a~partition of unity~\cite[3.1.13]{Fed69} one can
    decompose $g$ into a~finite sum $g = \sum_{i=1}^{N} g_i$, where $g_i \in
    \VF(U)$ is supported in some ball contained in~$U$ for each $i \in \{
    1,2,\ldots, N \}$.  Recalling that $\delta_F V$ is linear we see that there
    exists an $i \in \{1,2,\ldots,N\}$ such that $\delta_F V(g_i) \ne 0$. Set $h
    = g_i$ and $\varphi_t(x) = x + th(x)$ for $x \in U$ and $t$ in some
    neighbourhood of~$0$ in~$\R$. Clearly $\varphi_t \in \adm{U}$ is an
    injective admissible map whenever $|t|$ is small enough. Replacing possibly
    $h$ with $-h$ we shall assume that $\delta_F V(h) < 0$. Then there exists
    $t_0 > 0$ such that $\Phi_F( (\varphi_t)_\# V ) < \Phi_F(V) = \mu$ for $t
    \in (0,t_0]$. Set $\psi = \varphi_{t_0}$.

    In case $V = \var{\vdim}(S)$ for some $(\HM^{\vdim},\vdim)$~rectifiable set
    $S \in \mathcal C$, we have
    \begin{displaymath}
        \mu = \Phi_F(V) > \Phi_F( \psi_\# V ) = \Phi_F( \psi\lIm S \rIm ),
    \end{displaymath}
    which contradicts the definition of~$\mu$.

    In the other case, since $\psi_{\#} : \Var{\vdim}(U) \to \Var{\vdim}(U)$ is
    continuous and $V$ equals the limit $\lim_{j \to \infty} \var{\vdim}(S_j
    \cap U)$, we have also $\psi_{\#}V = \lim_{j \to \infty} \psi_{\#}
    \var{\vdim}(S_j \cap U)$.  For $j \in \nat$ we set $\bar S_j = \mathcal
    U(S_j \cap U)$ and $\hat S_j = \mathcal R(S_j \cap U)$ to obtain
    \begin{multline*}
        \mu > \lim_{j \to \infty} \Phi_F(\psi_{\#} \var{\vdim}(S_j \cap U))
        \ge \lim_{j \to \infty} \Phi_F(\psi_{\#} \var{\vdim}(\hat S_j)) 
        = \lim_{j \to \infty} \Phi_F(\var{\vdim}(\psi \lIm \hat S_j \rIm)) 
        \\
        = \lim_{j \to \infty} \Phi_F(\psi \lIm S_j \cap U \rIm) 
        - \Phi_F(\psi \lIm \bar S_j \rIm) 
        \,.
    \end{multline*}
    Since $\lim_{j \to \infty} \HM^{\vdim}(\bar S_j) = 0$, we see that $\mu >
    \lim_{j \to \infty} \Phi_F(\psi \lIm S_j \cap U \rIm)$ which contradicts the
    definition of~$\mu$.
\end{proof}

\begin{theorem}
    \label{thm:wbc-exist}
    Assume $U$, $\mathcal C$, $F$, $\mu$, $V$, $S$, and~$\{ S_i : i \in \nat\}$
    are as in Theorem~\ref{thm:FK} and that $F \in \FwBC$. Then
    \begin{enumerate}
    \item
        \label{i:wbc:tan}
        $T = \Tan^{\vdim}(\|V\|,x)$ for $V$ almost all $(x,T)$.
    \item
        \label{i:wbc:dens}
        $\density^{\vdim}(\|V\|,x) = 1$ for $\|V\|$ almost all $x$.
    \end{enumerate}
    In~particular, $V = \var{\vdim}(S)$.
\end{theorem}

\begin{proof}
    \textit{Proof of~\ref{i:wbc:tan}.} Employing Lemma~\ref{lem:fst-var-mini} together
    with~\cite[2.3, 2.4]{DDG2016rect}
    and Theorem~\ref{thm:FK}\ref{i:FK:rect}\ref{i:FK:var-lim}\ref{i:FK:min}\ref{i:FK:abs-cont}\ref{i:FK:unrect}
    we see that for $\|V\|$ almost all~$x$ and all $W \in \VarTan(V,x)$ there
    exists a Radon probability measure~$\sigma$ over~$\grass{\adim}{\vdim}$ such
    that
    \begin{gather}
        \label{eq:wbc:x-cond1}
        \Tan^{\vdim}(\|V\|,x) = T \in \grass{\adim}{\vdim} \,,
        \quad
        \density^{\vdim}(\|V\|,x) = \vartheta \in [1,\infty) \,,
        \\
        \label{eq:wbc:x-cond2}
        W = \vartheta (\HM^{\vdim} \restrict T) \times \sigma \,,
        \quad \text{and} \quad
        \delta_{F^x} W = 0 \,.
    \end{gather}
    Since $F \in \FwBC$ it follows that $\VarTan(V,x) = \{
    \density^{\vdim}(\|V\|,x)\var{\vdim}(\Tan^{\vdim}(\|V\|,x)) \}$ for $\|V\|$
    almost all~$x$ which proves~\ref{i:wbc:tan}.

    \textit{Proof of~\ref{i:wbc:dens}.} Let $T \in \grass{\adim}{\vdim}$ and
    $\vartheta \in [1,\infty)$
    satisfy~\eqref{eq:wbc:x-cond1}\eqref{eq:wbc:x-cond2}, and $x \in U$ be such
    that Theorem~\ref{thm:FK}\ref{i:FK:unrect}\ref{i:FK:Tan} hold. Without loss of
    generality we shall assume $x = 0$. Assume, by contradiction, that
    $\vartheta > 1$.  Define
    \begin{displaymath}
        \delta_r = \sup\left\{
            \frac{\dist(x,T)}{|x|} : x \in \spt \|V\| \cap \oball{x}{2r} \without \{ 0 \}
        \right\}
        \quad \text{for $r \in (0,\infty)$} \,.
    \end{displaymath}
    From Theorem~\ref{thm:FK}\ref{i:FK:Tan}, we see that $\delta_r \downarrow 0$ as $r
    \downarrow 0$. Set $\varepsilon_r = 12 \delta_r^{1/2}$. For $r \in (0,1)$
    let $f_r, h_r \in \cnt^{\infty}(\R,[0,1])$ be such that
    \begin{gather}
        f_r(t) = 1 \quad \forall t \le 1 - \varepsilon_r \,,
        \quad
        f_r(t) = 0 \quad \forall t \ge 1 - \tfrac 12 \varepsilon_r \,,
        \quad 
        |f_r'(t)| \le 4/\varepsilon_r \quad \forall t \in \R \,,
        \\
        h_r(t) = 1 \quad \forall t \le 2 \delta_r \,,
        \quad
        h_r(t) = 0 \quad \forall t \ge 3 \delta_r \,,
        \quad 
        |h_r'(t)| \le 2/\delta_r \quad \forall t \in \R \,.
    \end{gather}
    For $r \in (0,1)$ we define $p_r \in \cnt^{\infty}(\R^\adim,\R^\adim)$ by the
    formula
    \begin{gather}
        p_r(x) = \project{T}(x) + \bigl( 1- f_r(|\project{T}(x)|) h_r(|\perpproject{T}(x)|) \bigr) \perpproject{T}(x)
        \quad \text{for $x \in \R^\adim$} \,.
    \end{gather}
    Clearly $p_r \in \adm{U}$ for $r \in (0,1)$ small enough. Note also that
    \begin{gather}
        \begin{aligned}
            p_r(x) &= x 
            &&\text{for $x \in \R^{\vdim} \without ((T \cap \cball 0{1-\varepsilon_r/2}) + \cball 0{3\delta_r}) \subseteq \R^{\vdim} \without \oball 01$}
            \,,
            \\
            p_r(x) &= \project{T}x
            &&\text{for $x \in (T \cap \cball 0{1-\varepsilon_r}) + \cball 0{2\delta_r}$}
            \,,
        \end{aligned}
        \\
        \label{eq:wbc:Lip-p}
        \Lip  p_r \le 8 + 12 \frac{\delta_r}{\varepsilon_r} \le 8 +  \delta_r^{1/2} \le 9
        \quad \text{for $r \in (0,1)$} \,.
    \end{gather}
    Set $A_r = \cball 01 \without \oball 0{1 - \varepsilon_r}$ and $\tilde p_r =
    \scale{r} \circ p_r \circ \scale{1/r}$. Let $C \in \VarTan(V,0)$.
    By~\cite[3.4(2)]{All72} and~\ref{i:wbc:tan} we get
    \begin{equation}
        \label{eq:wbc:vartan}
        C = \lim_{r \downarrow 0} (\scale{1/r})_{\#}V 
        = \lim_{r \downarrow 0} \lim_{i \to \infty} \var{\vdim}(\scale{1/r} \lIm S_i \rIm)
        = \vartheta \var{\vdim}(T) \,;
    \end{equation}
    Hence, we have $\|C\|(\Bdry{\cball 01}) = 0$, which implies that
    \begin{displaymath}
        \lim_{r \downarrow 0} \lim_{i \to \infty} r^{-\vdim} \HM^{\vdim}( \scale{r} \lIm A_r \rIm \cap S_i ) = 0 \,.
    \end{displaymath}
    In particular, employing~\eqref{eq:wbc:Lip-p},
    \begin{equation}
        \label{eq:wbc:errors}
        \lim_{r \downarrow 0} \lim_{i \to \infty} r^{-\vdim} \Phi_F(\scale{r} \lIm A_r \rIm \cap S_i) = 0 
        \quad \text{and} \quad
        \lim_{r \downarrow 0} \lim_{i \to \infty} r^{-\vdim} \Phi_F(\tilde p_r \lIm  \scale{r} \lIm A_r \rIm \cap S_i \rIm) = 0 \,.
    \end{equation}
    For $r \in (0,1)$ and $i \in \nat$ we have
    \begin{multline}
        \label{eq:wbc:contr}
        \Phi_F(\tilde p_r \lIm S_i \cap U \rIm)
        = \Phi_F(S_i \cap U) 
        - \Phi_F(S_i \cap \cball 0{(1 - \varepsilon_r)r})
        \\
        + \Phi_F(\tilde p_r \lIm S_i \cap \cball 0{(1 - \varepsilon_r)r} \rIm)
        - \Phi_F(S_i \cap \scale{r} \lIm A_r \rIm) 
        + \Phi_F(\tilde p_r \lIm S_i \cap \scale{r} \lIm A_r \rIm \rIm) \,.
    \end{multline}
    Since $\lim_{i \to \infty} \Phi_F(S_i \cap U) = \mu$, taking into
    account~\eqref{eq:wbc:errors}, to reach a~contradiction it suffices to show
    that
    \begin{equation}
        \label{eq:wbc:pre-claim}
        \lim_{r \downarrow 0} \lim_{i \to \infty}
        r^{-\vdim} \Phi_F(\tilde p_r \lIm S_i \cap \cball 0{(1 - \varepsilon_r)r} \rIm) 
        - r^{-\vdim} \Phi_F(S_i \cap \cball 0{(1 - \varepsilon_r)r}) < 0 \,.
    \end{equation}
    For $i \in \nat$ and $r \in (0,1)$ we define
    \begin{displaymath}
        S_{r,i} = \scale{1/r} \lIm S_i \rIm  \cap \cball 01  \,,
        \quad
        F_r = \scale{r}^{\#} F \,,
        \quad \text{and} \quad
        \hat S_{r,i} = \mathcal R(S_{r,i}) \,.
    \end{displaymath}
    Observe that, using~\eqref{eq:wbc:errors} and Theorem~\ref{thm:FK}\ref{i:FK:unrect},
    claim~\eqref{eq:wbc:pre-claim} will follow from
    \begin{equation}
        \label{eq:wbc:claim}
        \lim_{r \downarrow 0} \lim_{i \to \infty} \Phi_{F_r}(\project T \lIm \hat S_{r,i} \rIm) -  \Phi_{F_r}( \hat S_{r,i}) < 0 \,.
    \end{equation}
    In order to prove~\eqref{eq:wbc:claim}, we observe that~\eqref{eq:wbc:vartan} implies
    \begin{displaymath}
        \lim_{r \downarrow 0} \lim_{i \to \infty} \int_{\cball 01} \| \project P - \project T\| \ud \var{\vdim}(\hat S_{r,i})(x,P) = 0 \,.
    \end{displaymath}
    Since $F$ is continuous, we obtain also
    \begin{equation}
        \label{eq:wbc:F-tilt}
        \lim_{r \downarrow 0} \lim_{i \to \infty} \int_{\cball 01} |F(z,P) - F(z,T)| \ud \var{\vdim}(\hat S_{r,i})(x,P) = 0 
        \quad \text{for any $z \in \R^\adim$} \,.
    \end{equation}
    We then estimate
    \begin{multline*}
        \Phi_{F_r}(\project T \lIm \hat S_{r,i} \rIm) -  \Phi_{F_r}( \hat S_{r,i}) 
        = \int_{\project T \lIm \hat S_{r,i} \rIm} F_r(y,T) \ud \HM^{\vdim}(y)
        - \int F_r(x,P) \ud \var{\vdim}(\hat S_{r,i})(x,P)
        \\
        \le \int_{\project T \lIm \hat S_{r,i} \rIm} F_r(0,T) \ud \HM^{\vdim}(y)
        - \int F_r(0,T) \ud \var{\vdim}(\hat S_{r,i})
        \\
        + \int_{\project T \lIm \hat S_{r,i} \rIm} |F_r(y,T) - F_r(0,T)| \ud \HM^{\vdim}(y)
        \\
        + \int |F_r(0,T) -  F_r(0,P)| + |F_r(0,P) -  F_r(x,P)| \ud \var{\vdim}(\hat S_{r,i})(x,P) \,.
    \end{multline*}
    Using continuity of~$F$ and~\eqref{eq:wbc:F-tilt}, we see that the last two
    terms converge to zero when we first take the limit with $i \to \infty$
    and then with $r \downarrow 0$. Therefore, 
    \begin{multline*}
        \lim_{r \downarrow 0} \lim_{i \to \infty}
        \Phi_{F_r}(\project T \lIm \hat S_{r,i} \rIm) -  \Phi_{F_r}( \hat S_{r,i})
        \\
        =
        \lim_{r \downarrow 0} \lim_{i \to \infty}
        \int_{\project T \lIm \hat S_{r,i} \rIm} F_r(0,T) \ud \HM^{\vdim}(y)
        - \int F_r(0,T) \ud \var{\vdim}(\hat S_{r,i})(x,P)
        \\
        = \lim_{r \downarrow 0} \lim_{i \to \infty}
        F_r(0,T) \bigl( 
        \HM^{\vdim} ( \project T \lIm \hat S_{r,i} \rIm ) 
        - \HM^{\vdim}( S_{r,i} )
        \bigr)
        \le \unitmeasure{\vdim} F_r(0,T) (1 - \vartheta) = - \kappa < 0 \,.
    \end{multline*}
    Thus, we have proved~\eqref{eq:wbc:claim}, which in turn
    implies~\eqref{eq:wbc:pre-claim}. Hence, recalling~\eqref{eq:wbc:contr}, we can
    choose $r \in (0,1)$ so that for all big enough $i \in \nat$
    \begin{displaymath}
        \Phi_F(\tilde p_r \lIm S_i \cap U \rIm) - \Phi_F(S_i \cap U) 
        < - \tfrac 12 \kappa r^{\vdim} \,.
    \end{displaymath}
    Up to choosing a bigger $i \in \nat$, we get~$\Phi_F(\tilde p_r \lIm S_i
    \cap U \rIm) < \mu$, which contradicts the definition of~$\mu$.
\end{proof}

\section{Equivalence of $\FBC$ and $\FAC$}
\label{equiv}
In this section we prove that the new condition $\FBC$ can be used in place
of~$\FAC$. First we prove a small lemma.

\begin{lemma}
    \label{lem:BC:k-ge-d}
    Let $F$ be an integrand of class~$\cnt^{1}$, $x \in \R^{\adim}$, $F \in
    \mathrm{BC}_x$, $\mu$ be a probability measure over $\grass{\adim}{\vdim}$,
    $k \in \nat$, $T \in \grass{\adim}{k}$, $W = (\HM^{k} \restrict T) \times
    \mu$. Then
    \begin{displaymath}
        \delta_{F^x} W = 0 \quad \implies \quad k \ge \vdim \,.
    \end{displaymath}
\end{lemma}

\begin{proof}
    If $\vdim = \adim$, then $\grass{\adim}{\vdim}$ contains only one element so
    there is only one probability measure over $\grass{\adim}{\vdim}$ and the
    conclusion readily follows.

    Assume $1 \le \vdim < \adim$ and $k < d$. Choose $R \in \grass{\adim}{\vdim
      - k}$ such that $R \perp T$ and set $V = (\HM^{\vdim} \restrict (T+R))
    \times \mu$. We get
    \begin{multline}
        \delta_{F^x} V(g)
        = \int_R \int_T \int_{\grass{\adim}{\vdim}}
        B_F(u+v,S) \bullet \uD g(x)
        \ud \mu(S) \ud \HM^k(u) \ud \HM^{\vdim - k}(v)
        \\
        = \int_R \delta_{F^x}W(g(v + \cdot)) \ud \HM^{\vdim - k}(v) = 0
        \quad \text{for $g \in \VF(\R^{\adim})$} \,.
    \end{multline}
    Thus, $\delta_{F^x} V = 0$ and, since $F \in \mathrm{BC}_x$, we obtain $\mu
    = \Dirac{T + R}$. Since $R$ was chosen arbitrarily from
    $\grass{\adim}{\vdim} \cap \{ R : R \perp T \} \simeq \grass{\adim -
      k}{\vdim - k}$ which contains more than one element, we reach
    a~contradiction.
\end{proof}

\begin{lemma}
    \label{lem:AC-eq-BC}
    Let $x \in \R^\adim$. We have
    \begin{math}
        \FAC_x = \FBC_x \,.
    \end{math}
\end{lemma}

\begin{proof}
    {\emph{Step 1}} We first prove that $\FAC_x \subseteq \FBC_x$. Let $F \in
    \FAC_x$, $\mu$ be a~Radon probability measure over~$\grass{\adim}{\vdim}$,
    and $T \in \grass{\adim}{\vdim}$. We define the varifold
    \begin{displaymath}
        W = (\HM^{\vdim} \restrict T) \times \mu \in \Var{\vdim}(\R^{\adim}) \,.
    \end{displaymath}
    Assume that $ \delta_{F^x} W = 0$. We will show that $\mu=\Dirac{T}$, i.e.,
    that $F \in \FBC_x$. By the very definition of anisotropic first variation,
    we deduce that for every test vector field $g \in \VF(\R^{\adim})$
    \begin{multline}
        \label{eq:deltaWg}
        0 = \delta_{F^x} W (g) 
        = \int B_F(x,S) \bullet \uD g(y) \ud W(y,S)
        \\
        = \int \int B_F(x,S) \bullet \uD g(y) \ud (\HM^{\vdim} \restrict T)(y) \ud \mu(S) 
        = \int A_x(\mu) \bullet \uD g(y) \ud (\HM^{\vdim} \restrict T)(y) \,.
    \end{multline}
    Let $e_1,\ldots,e_{\adim-\vdim}$ be an orthonormal basis of $T^{\perp}$. For
    any $\varphi \in \dspace{T}{\R}$, $i,j \in \{1,2,\ldots,\adim-\vdim\}$, we
    can find $g \in \VF(\R^{\adim})$ such that
    \begin{displaymath}
        g(y) = \varphi(\project T y) (y \bullet e_i) e_j
        \quad \text{whenever $y \in (T + \cball 01)$} \,;
    \end{displaymath}
    hence, equation~\eqref{eq:deltaWg} yields
    \begin{equation*}
        \int \varphi(y) A_x(\mu)e_i \bullet e_j \ud (\HM^{\vdim} \restrict T)(y) = 0 \qquad
        \text{for all $\varphi \in \dspace{T}{\R}$ and $i,j \in \{1,2,\ldots,\adim-\vdim\}$} \,,
    \end{equation*}
    which shows that $ T^{\perp} \subseteq \ker A_x(\mu)$. Since $\dim
    T^{\perp}=\adim-\vdim$, we get $\dim \ker A_x(\mu) \geq \adim-\vdim$.
    By~Definition \ref{def:AC}\ref{i:ac:a} we obtain $\adim-\vdim \le \dim \ker
    A_x(\mu) \le \adim - \vdim$, so it follows from
    Definition~\ref{def:AC}\ref{i:ac:b} that $\mu = \Dirac S$ for some $S \in
    \grass{\adim}{\vdim}$. Then
    \begin{displaymath}
        A_x(\mu) = B_F(x,S) \,.
    \end{displaymath}
    Directly from the definition of $B_F(x,S)$ it follows that $S^{\perp}
    \subseteq \ker B_F(x,S)$. Therefore, since $\dim \ker B_F(x,S) = \adim -
    \vdim$ and $T^{\perp} \subseteq \ker B_F(x,S) = \ker A_x(\mu)$, we see that
    $S = T$, which shows that $F \in \FBC_x$.

    {\emph{Step 2}} We prove now that $\FBC_x \subseteq \FAC_x$. Assume $F \in
    \FBC_x$. Given a~Radon probability measure~$\mu$ over~$\grass{\adim}{d}$,
    we~define
    \begin{displaymath}
        T = \im(A_x(\mu)^*) \,,
        \quad
        k = \dim T \,,
        \quad
        V = (\HM^{k} \restrict T) \times \mu \in \Var{\vdim}(\R^{\adim}) \,.
    \end{displaymath}
    Note that $T^{\perp} = [\im(A_x(\mu)^*)]^\perp = \ker A_x(\mu)$. Thus,
    similarly as in~\eqref{eq:deltaWg}, we get that for every $g \in
    \VF(\R^{\adim})$
    \begin{equation*}
        \delta_{F^x} V(g) 
        = A_x(\mu) \bullet \int \uD (g \circ \project T)(y) \ud (\HM^{k} \restrict T)(y)
        \\
        + \int A_x(\mu) \bullet \bigl( \uD g(y) \circ \perpproject T \bigr) \ud (\HM^{k} \restrict T)(y)
        = 0 \,.
    \end{equation*}
    By Lemma~\ref{lem:BC:k-ge-d}, we obtain $\dim T = k \geq \vdim$ and conclude
    that
    \begin{displaymath}
        \dim \ker A_x(\mu) = n - \dim T \leq \adim - \vdim \,,
    \end{displaymath}
    which is Definition \ref{def:AC}\ref{i:ac:a}. Moreover, if $\dim \ker
    A_x(\mu) = \adim - \vdim$, then $\dim T = \vdim$ and we can apply
    Definition~\ref{def:BC} to the varifold~$V$ and deduce that $\mu =
    \Dirac{T}$, which is precisely Definition~\ref{def:AC}\ref{i:ac:b}.
\end{proof}

\section{The inclusion $\FwBC \subseteq \FAE(\mathcal P)$}
\label{sec:wbc-in-ae}

In this section we work with cubical test pairs $(S,Q)$, where $Q$ is now
a~$\vdim$-dimensional cube; see Definition~\ref{def:cube-test-pair}. Cubical
test pairs give rise to the same classes of Almgren elliptic integrands as the
test pairs defined in Definition~\ref{def:test-pair}; see
Remark~\ref{rem:ellipticcube}.

The main result is Theorem~\ref{thm:wBC-in-AE}, which shows that $\FwBC_x \subseteq
\FAE_x(\mathcal P)$ given $\mathcal P$ is closed under Lipschitz deformations
leaving the boundary fixed and under gluing together several rescaled copies of
an element of~$\mathcal{P}$; see Definition~\ref{def:good-test-pairs}.

The second closedness property for $\mathcal P$ is needed to be able to perform
a~``homogenization'' (one could also call it a~``blow-down'') argument. More
precisely, given a~minimiser~$P$ of~$\Phi_{F^x}$ in~$\{ R : (R,Q) \in \mathcal P
\}$ we construct the varifold~$W$, occurring in Definition~\ref{def:BC}, so that
$W \restrict Q \times \grass{\adim}{\vdim}$ is a~limit of a~sequence of
varifolds $\tilde W_N = \var{\vdim}(P_N)$, where $P_N$ is constructed, for $N
\in \nat$, by gluing together $2^{N\vdim}$ rescaled copies of~$P$. A crucial
observation is that $P_N$ has the same $\Phi_{F^x}$~energy as~$P$ which, in
turn, is a~minimiser of~$\Phi_{F^x}$ in~$\mathcal P$. This allows us to deduce
that $\delta_{F^x} W_N = 0$ using Lemma~\ref{lem:fst-var-mini}, provided $P_N$
is a~competitor (or a limit of competitors), i.e., if $(P_N,Q) \in \mathcal P$
for an appropriate choice of the cube~$Q$.

It is not at all obvious that Theorem~\ref{thm:wBC-in-AE} is valid with
$\mathcal P$ being the set of all cubical test pairs; see
Remark~\ref{rem:ctp-good-p}.  The~proof that such family~$\mathcal P$ has the
necessary closedness property requires some subtle topological arguments and is
postponed to~Section~\ref{sec:all-test-pairs-good}; see~\ref{thm:all-tp-good}.
\begin{definition}
    \label{def:cube-test-pair}
    Let $Q_0 = [-1,1]^{\vdim} \subseteq \R^{\vdim}$. We say that $(S,Q)$ is a
    \emph{cubical test pair} if there exists $p \in \orthproj{\adim}{\vdim}$
    such that
    \begin{gather}
        Q = p^* \lIm Q_0 \rIm \,,
        \quad
        B = p^* \lIm \Bdry{Q_0} \rIm \,,
        \quad
        S \subseteq \R^\adim \text{ is compact and $(\HM^d,d)$~rectifiable} \,,
        \\
        \label{eq:cube-non-retr}
        \text{$f \lIm S \rIm \ne B$
          for all $f : \R^\adim \to \R^\adim$
          satisfying $\Lip f < \infty$
          and $f(x) = x$ for $x \in B$} \,.
    \end{gather}
\end{definition}
\begin{remark}
    \label{rem:ellipticcube}
    In the rest of the paper we will work for simplicity on cubical
      test pairs, but it's worth to remark that the two notions are perfectly
      equivalent for our purposes. Indeed, if we denote with $\mathcal P_1$ the
      set of rectifiable test pairs and with $\mathcal P_2$ the set of cubical
      test pairs, then we easily verify that for every $F$ being
      a~$\cnt^0$~integrand and $x \in \R^\adim $, it holds $\FAE_x(\mathcal P_1) =
      \FAE_x(\mathcal P_2)$ and $\FAUE_x(\mathcal P_1) = \FAUE_x(\mathcal P_2)$.
      To show this, we denote $\rho = \sqrt d$ and $Q_0 = [-1,1]^d$.
    
    Given $(S,Q) \in \mathcal P_2$, we find $p \in \orthproj{\adim}{\vdim}$ such
    that $Q = p^*\lIm Q_0\rIm$ and construct $(R,D) \in \mathcal P_1$ by setting
    \begin{displaymath}
        T = \im p^* \,,
        \quad
        D = T \cap \cball 01 \,,
        \quad
        \bar D = \scale{\rho}\lIm D \rIm \,,
        \quad
        \bar R = S \cup (\bar D \without Q) \,, 
        \quad
        R = \scale{1/\rho}\lIm \bar R \rIm 
        \,.
    \end{displaymath}
    Then
    \begin{displaymath}
        \rho^{d} \bigl( \Phi_{F^x}(R) - \Phi_{F^x}(D) \bigr)
        =  \Phi_{F^x}(\bar R) - \Phi_{F^x}(\bar D)
        = \Phi_{F^x}(S) - \Phi_{F^x}(Q) \,.
    \end{displaymath}
    Given $(R,D) \in \mathcal P_1$ we choose $p \in \orthproj{\adim}{\vdim}$
    such that $D \subseteq \im p^*$ and construct $(S,Q) \in \mathcal P_2$ by
    setting
    \begin{displaymath}
        Q = p^* \lIm Q_0 \rIm \,,
        \quad 
        \bar Q = \scale{\rho}\lIm \bar Q \rIm \,,
        \quad
        \bar S = R \cup (\bar Q  \without D) \,,
        \quad
        S = \scale{1/\rho}\lIm \bar S \rIm \,.
    \end{displaymath}
    Then
    \begin{displaymath}
        \rho^{d} \bigl( \Phi_{F^x}(S) - \Phi_{F^x}(Q) \bigr)
        =  \Phi_{F^x}(\bar S) - \Phi_{F^x}(\bar Q)
        = \Phi_{F^x}(R) - \Phi_{F^x}(D) \,.
    \end{displaymath}
    Therefore, $\FAE_x(\mathcal P_1) = \FAE_x(\mathcal P_2)$ and
    $\FAUE_x(\mathcal P_1) = \FAUE_x(\mathcal P_2)$.
\end{remark}

\begin{definition}
    \label{def:multiplication}
    Let $Q$ be a~$\vdim$-dimensional cube in~$\R^{\adim}$ (see
    Definition~\ref{def:cube}), and $X \subseteq \R^{\adim}$. We say
    that~$(Y,Q)$ is a~\emph{multiplication} of~$(X,Q)$ if there exist $k \in
    \natp$ and a~finite set $\mathcal A$ of $\vdim$-dimensional cubes
    in~$\R^{\adim}$ of side-length $\side{Q}/k$ such that
    \begin{gather}
        Q = \tbcup \mathcal A \,,
        \quad
         \cInt K \cap \cInt L = \varnothing \, \, \forall K\neq L \in \mathcal A  \,,
        \\
        Y = \tbcup \bigl\{ \trans{\centre K} \circ \scale{1/k} \circ \trans{-\centre Q} \lIm X \rIm : K \in \mathcal A \bigr\}
        \,.
    \end{gather}
\end{definition}

\begin{remark}
    \label{rem:multiplication}
    Observe that a multiplication $(Y,Q)$ of $(X,Q)$ is uniquely determined by
    the parameter $k$ occurring in Definition~\ref{def:multiplication}. Thus, we
    may define the \emph{$k$-mul\-ti\-pli\-cation of~$(X,Q)$} to be exactly~$(Y,Q)$.
\end{remark}

\begin{definition}
    \label{def:good-test-pairs}
    We say that a set $\mathcal Q$ of pairs of subsets of~$\R^{\adim}$ is
    a~\emph{good family} if
    \begin{enumerate}
    \item all elements of $\mathcal Q$ are cubical test pairs;
    \item if $(X,Q) \in \mathcal Q$, $N \in \nat$, and $(Y,Q)$ is the
        $2^N$-multiplication of~$(X,Q)$, then $(Y,Q) \in \mathcal Q$;
    \item if $(X,Q) \in \mathcal Q$, $f : \R^{\adim} \to \R^{\adim}$ is
        Lipschitz, and $f(x) = x$ for $x \in \cBdry{Q}$, then $(f\lIm X \rIm,Q)
        \in \mathcal Q$.
    \end{enumerate}
\end{definition}

\begin{remark}
    \label{rem:ctp-good-p}
    It is plausible that the set of all cubical test pairs is a good family and,
    indeed, in Section~\ref{sec:all-test-pairs-good} we prove it is. However,
    this is not at all obvious.

    Consider the Adams' surface; see~\cite[Example~8 on p.~81]{Rei60}.
    The M{\"o}bius strip~$M$ and the triple M{\"o}bius strip~$T$ are both
    homotopy equivalent to the~$1$-dimensional sphere and both can be
    continuously embedded in some~$\R^{\adim}$ so that $(M,Q)$ and $(T,Q)$
    become cubical test pairs, where $Q = [0,1]^{2} \times \{0\}^{\adim-2}$.
    However, if one puts~$M$ and~$T$ side by side touching only along one
    $1$-dimensional face of~$Q$, then one obtains the Adams' surface~$A$, which
    retracts onto its boundary. This, as explained in~\cite[Example~8 on
    p.~81]{Rei60}, is a~consequence of the fact that the inclusion of the
    boundary of~$M$ into~$M$ has degree~$2$, the inclusion of the boundary
    of~$T$ into~$T$ has degree~$3$, these numbers are relatively prime, and~$A$
    is homotopy equivalent to the wedge sum (a.k.a. ``bouquet'';
    see~\ref{def:wedge-sum}) of two circles so, defining $f : A \to \sphere{1}$
    to be of degree $-1$ on~$M$ and of degree~$1$ on~$T$, we get a~map such that
    $f \circ j$ is of degree~one, where $j : \sphere{1} \to A$ is
    a~parameterization of the boundary of~$A$. One can then
    construct a Lipschitz retraction of $A$ onto its boundary;
    see~\ref{lem:thick-retr}. Luckily for us, the situation is different if one
    puts together many copies of \emph{the~same} set~$X$. We prove
    in~\ref{cor:deg-gcd} that if $(X,Q)$ is a cubical test pair, then one
    cannot have two maps $f,g : X \to \cBdry{Q}$ such
    that~$\deg(f|_{\cBdry{Q}})$ and~$\deg(g|_{\cBdry{Q}})$ are relatively prime.
\end{remark}

Before stating and proving the main theorem of this section, we need
the following lemma, which, roughly speaking, will be used as an
{\emph{almost}} uniqueness result for minimizers of the area functional in the
class of cubical test pairs:
\begin{lemma}\label{lem:uniqueness}
    Given a cubical test pair $(R,Q)$ as in Definition \ref{def:cube-test-pair}
    and $x \in \R^\adim$. If
    \begin{equation}\label{uniqueness0}
        \Phi_{F^x}(R) < \Phi_{F^x}(Q) \,,
    \end{equation}
    then 
    \begin{equation}\label{uniqueness}
        \HM^{\vdim}(R) > \HM^{\vdim}(Q) \,.
    \end{equation}
\end{lemma}
\begin{proof}
    Assume by contradiction that \eqref{uniqueness} does not hold. Thus in particular
    \begin{equation}\label{uniqueness1}
        \HM^{\vdim}(R \cap (Q \times \R^{n-d})) \leq  \HM^{\vdim}(R) \leq \HM^{\vdim}(Q) \,.
    \end{equation}
    Denoting with $T$ the $d$-plane containing $Q$, we observe that 
    \begin{equation}\label{uniqueness2}
        \HM^{\vdim}(R \cap (Q \times \R^{n-d})) \geq \HM^{\vdim}(\project T(R \cap (Q \times \R^{n-d}))) \geq \HM^{\vdim}(Q) \,,
    \end{equation}
    otherwise there would exist a $d$-dimensional open ball $B\subset Q$ such that
    \begin{equation}\label{uniqueness3}
        (B\times \R^{n-d}) \cap R = \emptyset \,.
    \end{equation}
    Since $R$ is compact, then \eqref{uniqueness3} would imply the existence of
    $f : \R^\adim \to \R^\adim$ satisfying $\Lip f < \infty$ and $f(x) = x$ for
    $x \in \partial_c Q$, such that $f \lIm R \rIm = \partial_c Q$, which would
    contradict the property of $(R,Q)$ being a cubical test pair.  By
    \eqref{uniqueness2} and the area formula (a.f.)~\cite[3.2.20]{Fed69},
    we~compute
    \begin{multline}
        \label{uniqueness4}
        \HM^{\vdim}(Q)
        \overset{\eqref{uniqueness2}}{\leq} \HM^{\vdim}(\project T(R \cap (Q \times \R^{n-d})))
        \leq \int_{Q}\HM^{0}(\project T^{-1}(y)\cap R )\ud \HM^{\vdim}(y)\\
        \overset{(a.f.)}{=} \int_{R \cap (Q \times \R^{n-d})}\ap J_{\vdim} \project T(y) \ud \HM^{\vdim}(y)
        \leq \HM^{\vdim}(R \cap (Q \times \R^{n-d}))
        \overset{\eqref{uniqueness1}}{\leq} \HM^{\vdim}(Q) \,.
    \end{multline}
    Then the inequalities in \eqref{uniqueness4} are all equality, which implies
    that $\ap J_{\vdim} \project T(y)= 1$ for $\HM^{\vdim}$-a.e. $y \in R \cap (Q \times
    \R^{n-d})$. Hence,
    \begin{equation}\label{uniqueness5}
        \Tan^{\vdim}(\HM^{\vdim} \restrict R, y)=T,
        \qquad \mbox{for $\HM^{\vdim}$-a.e. $y \in  R \cap (Q \times \R^{n-d})$} \,.
    \end{equation}
    We can then compute the following chain of inequalities, which provides a contradiction
    \begin{align*}
      \Phi_{F^x}(Q) &= \int_{Q}F^x(T)\ud \HM^{\vdim}(y)
                      \overset{\eqref{uniqueness2}}{\leq} \int_{R \cap (Q \times \R^{n-d})}F^x(T)\ud \HM^{\vdim}(y)
      \\
                    &\overset{\eqref{uniqueness5}}{\leq} \Phi_{F^x}(R \cap (Q \times \R^{n-d}))
                      \leq \Phi_{F^x}(R) \overset{\eqref{uniqueness0}}{<} \Phi_{F^x}(Q) \,.
                      \qedhere
    \end{align*}
\end{proof}
We can finally prove the following:
\begin{theorem}
    \label{thm:wBC-in-AE}
    Assume $x \in \R^{\adim}$ and $\mathcal P$ is a~good family
    (cf. Definition~\ref{def:good-test-pairs}). Then $\FwBC_x \subseteq
    \FAE_x(\mathcal P)$.
\end{theorem}

\begin{proof}
    We proceed by contradiction. Assume $F \in \FwBC_x \without \FAE_x(\mathcal
    P)$. Then there exists $(S,Q) \in \mathcal P$ such that
    \begin{displaymath}
        \HM^{\vdim}(S) > \HM^{\vdim}(Q)
        \quad \text{and} \quad
        \Phi_{F^x}(S) \leq \Phi_{F^x}(Q) \,.
    \end{displaymath} 
    Define
    \begin{displaymath}
        B = \cBdry{Q}
        \quad \text{and} \quad
        \mathcal C = \bigl\{ S : (S,Q) \in \mathcal P \bigr\} \,.
    \end{displaymath}
    Note that $\mathcal C$ is a good class in $\R^{\adim} \without B$ in
    the sense of Definition~\ref{def:goodclass}.
    
    Next, we employ Theorem~\ref{thm:wbc-exist} with~$F^x$ in place of~$F$ together
    with Theorem~\ref{thm:FK}\ref{i:FK:min}\ref{i:FK:rect}\ref{i:FK:compact} to find
    a~compact $(\HM^{\vdim},\vdim)$~rectifiable set $R \subseteq \R^{\adim}$
    such that
    \begin{displaymath}
        \Phi_{F^x}(R) = \inf \bigl\{ \Phi_{F^x}(P) : P \in \mathcal C \bigr\} \le \Phi_{F^x}(S) \le \Phi_{F^x}(Q) \,.
    \end{displaymath}
    Proceeding as in Lemma~\ref{lem:ctp-HD-limit} we see that $(R,Q)$ is a
    cubical test pair (may be not in $\mathcal P$). In case $\Phi_{F^x}(R) < \Phi_{F^x}(Q)$,
    by Lemma~\ref{lem:uniqueness} we get $\HM^{\vdim}(R) > \HM^{\vdim}(Q)$,
    and we set $P = R$. Otherwise, we have $\Phi_{F^x}(R) = \Phi_{F^x}(Q) =
    \Phi_{F^x}(S)$ and we set $P = S$. In any case, setting $V = \var{\vdim}(P)
    \in \Var{\vdim}(\R^{\adim})$ and using Lemma~\ref{lem:fst-var-mini}, we
    obtain
    \begin{equation}
        \label{eq:fst-var-P}
        \infty > \HM^{\vdim}(P) > \HM^{\vdim}(Q)
        \quad \text{and} \quad
        \delta_{F^x} V(g) = 0
        \quad \text{for $g \in \VF(\R^{\adim} \without B)$} \,.
    \end{equation}
    Let $p \in \orthproj{\adim}{\vdim}$ and $T \in \grass{\adim}{\vdim}$ be such
    that $p^*[Q_0] = Q \subseteq T$, where $Q_0 = [-1,1]^{\vdim}$.  For each $N
    \in \nat$ we define $P_N$ and $\mathcal A_N$ so that $(P_N,Q)$ is the
    $2^N$-mul\-ti\-pli\-cation of~$(P,Q)$ and $\mathcal A_N$ is the corresponding set
    of $\vdim$-dimensional cubes covering~$Q$ as in Definition~\ref{def:multiplication}.
    We also set
    \begin{gather}
        W_N = \sum_{v \in \integers^{\vdim}} \var{\vdim}(\trans{p^*(2v)} \lIm P_N \rIm) \in \Var{\vdim}(\R^{\adim})
        \\
        \text{and} \quad
        R_K = \trans{\centre K} \circ \scale{2^{-N+1}} \lIm P \rIm 
        \quad \text{for $K \in \mathcal A_N$}
        \,.        
    \end{gather}
    Observe that for $N \in \nat$ and $\rho \in (0,\infty)$ there are at most
    $\unitmeasure{\vdim} \bigl( \rho + \diam P \bigr)^{\vdim}$ translated copies
    of~$P_N$ in $\spt \|W_N\| \cap \cball 0{\rho}$; therefore,
    \begin{equation*}
        \measureball{\| W_N \|}{\cball 0\rho}
        \le \unitmeasure{\vdim} \bigl( \rho + \diam P \bigr)^{\vdim} \HM^{\vdim}(P_N)
        = \unitmeasure{\vdim} \bigl( \rho + \diam P \bigr)^{\vdim} \HM^{\vdim}(P)  
        \quad \text{for $\rho \in (0,\infty)$} \,.
    \end{equation*}
    So $W_N$ is a Radon measure and there exists a~subsequence $\{ W_{N_i} : i \in \nat
    \}$ which converges to some varifold~$W$ in~$\Var{\vdim}(\R^{\adim})$.
    Moreover, we have
    \begin{displaymath}
        R_K \subseteq T + \cball 0{2^{-N}\diam P}
        \quad \text{for $K \in \mathcal A_N$} 
        \quad \text{so $\spt \|W\| \subseteq T$}\,.
    \end{displaymath}
    Directly from the construction and by density of base~$2$ rational numbers
    in~$\R$, it follows also that~$W$ is translation invariant in~$T$, i.e.,
    $(\trans{v})_{\#} W = W$ for all $v \in T$. Hence, there exists $\vartheta
    \in (0,\infty)$ and a Radon probability measure $\mu$ over
    $\grass{\adim}{\vdim}$ such that
    \begin{displaymath}
        W = \vartheta (\HM^{\vdim} \restrict T) \times \mu 
        \quad \text{and} \quad
        \vartheta = \frac{\HM^{\vdim}(P)}{\HM^{\vdim}(Q)} > 1 \,.
    \end{displaymath}
    We define
    \begin{gather}
        \tilde W_N = \var{\vdim}(P_N)  \in \Var{\vdim}(\R^\adim)
        \quad \text{for $N \in \nat$}
        \quad \text{and} \quad
        \tilde W = \lim_{i \to \infty} \tilde W_{N_i}
        = \vartheta (\HM^{\vdim} \restrict Q) \times \mu \,.
    \end{gather}
    We also record that
    \begin{displaymath}
        \HM^{\vdim}(P_N) = \HM^{\vdim}(P) 
        \quad \text{and} \quad
        \Phi_{F^x}(P_N) = \Phi_{F^x}(P) \quad \text{for $N \in \nat$}\,,
    \end{displaymath}
    and since the supports of $\| \tilde W_N \|$ for $N \in \nat$ all lie in
    a~fixed compact set (cf. Remark~\ref{rem:cpt-spt-vari}) we also have
    \begin{equation}
        \label{eq:W-lim-PN}
        \Phi_{F^x}(\tilde W) = \lim_{i \to \infty} \Phi_{F^x}(\tilde
        W_{N_i}) = \lim_{i \to \infty} \Phi_{F^x}(P_{N_i}) = \Phi_{F^x}(P) \,.
    \end{equation}
    We claim that
    \begin{equation}
        \label{claimone}
        \delta_{F^x} W = 0.
    \end{equation}

    First we observe that this would immediately give a contradiction and
    conclude the proof. Indeed, since $F \in \FwBC_x$, we deduce
    from~\eqref{claimone} and Definition~\ref{def:BC} that $\mu = \Dirac{T}$. This, in
    turn, yields the following contradiction
    \begin{displaymath}
        \Phi_{F^x}(Q) 
        < \vartheta\Phi_{F^x}(Q) 
        = \Phi_{F^x}(\tilde W)
        \operatorname*{=}^{\eqref{eq:W-lim-PN}} \Phi_{F^x}(P) \leq \Phi_{F^x}(Q) \,.
    \end{displaymath}

    We are just left to prove the claim \eqref{claimone}. To this end, since $W$
    is invariant under translations in~$T$, it suffices to show that
    \begin{equation}
        \label{1}
        \delta_{F^x} \tilde W_N(g) = 0 \quad \text{for $N \in \nat$ and $g \in \VF(\R^{\adim} \without B)$} \,.
    \end{equation}

    If $P = S \in \mathcal C$, since $\mathcal C$ is a good family, then $P_N \in \mathcal C$ and $\tilde W_N =
    \var{\vdim}(P_N)$ and
    \begin{displaymath}
        \|\tilde W_N\|(\R^{\adim}) = \HM^{\vdim}(P) = \inf \{ \Phi_{F^x}(K) : K \in \mathcal C \}
        \quad \text{for $N \in \nat$} \,;
    \end{displaymath}
    hence, applying Lemma~\ref{lem:fst-var-mini}, we see that $\delta_{F^x}
    \tilde W_N(g) = 0$ for $g \in \VF(\R^{\adim} \without B)$ and $N \in \nat$.

    In case $P = R$, we use Theorem~\ref{thm:FK} to find a minimising sequence $\{ S_i
    \in \mathcal C : i \in \nat \}$ such that $\var{\vdim}(P) = V = \lim_{i \to
      \infty} \var{\vdim}(S_i \cap \R^{\adim} \without B)$. Defining $S_{i,N}
    \in \mathcal C$ so that $(S_{i,N},Q)$ is the $2^N$-multiplication
    of~$(S_i,Q)$ we get $\tilde W_N = \lim_{i \to \infty} \var{\vdim}(S_{i,N})$.
    Recalling Theorem~\ref{thm:FK}\ref{i:FK:var-lim}\ref{i:FK:min}\ref{i:FK:unrect} we
    may once again apply Lemma~\ref{lem:fst-var-mini} to see that also in this case
    $\delta_{F^x} \tilde W_N(g) = 0$ for $g \in \VF(\R^{\adim} \without B)$ and
    $N \in \nat$ so the proof is done.
\end{proof}

\section{Cubical test pairs form a good family}
\label{sec:all-test-pairs-good} 
Here we prove that the family of all cubical test pairs is good in the sense
of~\ref{def:good-test-pairs}. To our surprise the proof had to employ a few
sophisticated (yet classical) tools of algebraic topology. Given a cubical test
pair $(X,Q)$ and its $2^N$-multiplication $(Y,Q)$ we need to show that $S =
\cBdry Q$ is not a~Lipschitz retract of~$Y$, which is the same as showing that
there is no continuous map $f : Y \to S$ with $\deg(f|_S) = 1$;
cf.~\ref{lem:thick-retr}. This becomes a~topological problem of independent
interest. We first sketch the idea of the proof, highlighting the main points of
the argument.

Let $(X,Q)$ be a~cubical test pair. To~be able to use tools of algebraic
topology we need to pass from an arbitrary compact set~$X$ satisfying $0 <
\HM^{\vdim}(X) < \infty$ to an open set~$U$ containing~$X$ and~having homotopy
type of a~$\vdim$-dimensional CW-complex. We~achieve this by applying the
deformation theorem~\ref{thm:mae:dt} to~$X$, obtaining an open set~$U \subseteq
\R^{\adim}$ with $X \subseteq U$ and a~$\vdim$-dimensional cubical complex~$E
\subseteq U$ such that $\cBdry{Q} \subseteq E \subseteq U$ and $E$ is a~strong
deformation retract of~$U$; see~\ref{lem:cw-structure}. Moreover, we get that
$(U,E)$ is a~\emph{Borsuk pair}, i.e., has the homotopy extension property HEP;
see~\ref{def:hep} and~\ref{rem:hep-suff}, which will be a useful tool to get
suitable homotopy equivalences.

The topological part of the argument works as follows. Consider
a~$2$-mul\-ti\-pli\-cation~$(\tilde Y,Q)$ of~$(U,Q)$ and assume there exists
a~retraction $\tilde r : \tilde Y \to \cBdry{Q}$. Note that $\cBdry{Q}$ is
a~topological $(\vdim-1)$-dimensional sphere and set $S = \cBdry{Q}$.  Different
copies of~$\scale{1/2} \lIm U \without S \rIm$ may, in~general, intersect inside
$\tilde Y$. Thus, we define the \emph{lifted $2$-multiplication}~$(Y,Q)$
of~$(U,Q)$ in order to prevent this intersection and we observe that $\tilde r$
gives rise to a~retraction $r : Y \to S$; cf.~\ref{def:lifted-mult}. Next, we
consider the pairwise orthogonal affine $(\vdim-1)$-planes, lying in the affine
$\vdim$-plane spanned by~$Q$, parallel to the sides of~$Q$, and passing through
the center of~$Q$. We denote with~$R$ the union of these planes intersected
with~$Q$. Since~$R$ is contractible, by the aforementioned HEP, we deduce
that~$Y$ is homotopy equivalent to~$Y/R$ which, in turn, is homotopy equivalent
to the~wedge sum~$Z$ of~$2^{\vdim}$~copies of~$U$; see~\ref{def:wedge-sum}.
Let~$\Sigma$ be the wedge sum of $2^{\vdim}$~copies of~$S$, $\pi_i : \Sigma \to
S$ be projections onto particular components of~$\Sigma$, $\tau_i : S
\hookrightarrow \Sigma$ be inclusions of components, and $j : \Sigma
\hookrightarrow Z$ be the inclusion map; cf.~\ref{rem:wedge-proj}. The inclusion
$S \hookrightarrow Y$ composed with the homotopy equivalences yields a~map
$\alpha : S \to \Sigma \subset Z$ such that $\deg(\pi_i \circ \alpha) = 1$ for
all $i \in \{ 1,2, \ldots, 2^{\vdim} \}$.  In~particular, since
$\Sh_{\vdim-1}(\Sigma) \simeq \bigoplus_{i=1}^{2^{\vdim}} \Sh_{\vdim-1}(S) =
\integers^{2^{\vdim}}$ by~\cite[Corollary~2.25]{Hatcher2002}, we get
  \begin{equation}
      \label{eq:alpha-sum}
      \alpha_* = {\textstyle \sum_{i=1}^{2^{\vdim}} \tau_{i*} : \Sh_{\vdim-1}(S) \to \Sh_{\vdim-1}(\Sigma) } \,.
  \end{equation}
If~$\rho : Z \to S$ is obtained by composing the retraction~$r$ with the
homotopy equivalences, then $\deg(\rho \circ j \circ \alpha) = 1$. The following
homotopy commutative diagram presents the situation.
\begin{displaymath}
    \xymatrix{
      & & & & S \ar[r]^(0.3){\tau_i}
      & {\Sigma = \textstyle \bigvee_{i=1}^{2^{\vdim}} S} \ar[d]^j \ar[r]^(0.7){\pi_i}
      & S 
      \\
      S \ar[r] \ar@<1ex>@/_/[urrrrr]^{\alpha} 
      & Y \ar@/^/[rr] \ar@{}[rr]|{\approx} \ar@/_15pt/@<-5pt>[rrrrr]_r
      && Y/R \ar@/^/[rr] \ar@/^/[ll] \ar@{}[rr]|(.4){\approx}
      && Z = {\textstyle \bigvee_{i=1}^{2^{\vdim}} U}  \ar[r]^(.7){\rho} \ar@/^/[ll]
      & S
    }
\end{displaymath}
Recalling~\eqref{eq:alpha-sum} we see that~$1 = \deg(\rho \circ j \circ \alpha)
= \sum_{i=1}^{2^\vdim} m_i$, where $m_i = \deg(\rho \circ j \circ \tau_i)$.
Since~$Z$ is a wedge sum of copies of~\emph{the same} space~$U$, we
get~$2^{\vdim}$ maps $f_i : U \to S$ such that $\deg( f_i|_{S} ) = m_i$ and
$\sum_{i=1}^{2^{\vdim}} m_i = 1$.  The question now is whether there exists $g :
U \to S$ which induces the map
\begin{displaymath}
    \sum_{i=1}^{2^{\vdim}} f_{i*} : \Sh_{\vdim-1}(U) \to \Sh_{\vdim-1}(S) = \integers \,.
\end{displaymath}
If so, then $\deg(g|_{S})=1$ and $g$ yields a~retraction $U \to S$
by~\ref{lem:thick-retr}.

This is the point where we need to employ algebra and algebraic topology.
We~prove in~\ref{lem:hom-real} that if~$E$ is a~$\vdim$-dimensional CW-complex,
then any homomorphism $\zeta : \Sh_{\vdim-1}(E) \to \integers$ is induced by
some map $g : E \to S$. The cellular homology of~$E$ (which coincides with the
singular homology) is computed from the chain complex
$(C_k,\delta_k)_{k=0}^{\vdim}$, where the group of~$k$-dimensional
\emph{chains}~$C_{k}$ is the free abelian group generated by the $k$-dimensional
cells (or cubes) of~$E$. Observe that if $G$ is a torsion group (i.e. every
element has finite order), then there exists only one homomorphism $G \to
\integers$, namely, the one sending all elements of~$G$ to zero.  Therefore, we
do not lose any information by composing the~homomorphism~$\zeta$ with the
projection $p : \ker \delta_{\vdim-1} \twoheadrightarrow \ker \delta_{\vdim-1} /
\im \delta_{\vdim} = \Sh_{\vdim-1}(E)$, which yields a~homomorphism $\xi = \zeta
\circ p$ defined on~\emph{cycles}.  Since~$C_{\vdim-1}$ and~$C_{\vdim-2}$ are
\emph{free groups} (in~particular, projective $\integers$-modules), the group
$C_{\vdim-1}$ splits into a direct sum~$C_{\vdim-1} = \ker(\delta_{\vdim-1})
\oplus H$ and we can extend $\xi$~to all~\emph{chains} by setting $\xi|_H = 0$;
cf.~\ref{rem:free-groups}. Hence we can define~$g$ on any
$(\vdim-1)$-dimensional cell~$\sigma$ of~$E$ as~$g|_{\sigma} = h_{\sigma} \circ
\pi$, where $\pi : \sigma \twoheadrightarrow \sigma/\cBdry{\sigma} \simeq S$ and
$h_{\sigma} : S \to S$ is a~map of degree~$\xi(\sigma)$. The next step is to
extend~$g$ to all the $\vdim$-dimensional cells of~$E$. To this end we employ
the \emph{obstruction theory}, which is a~sophisticated version of the Brouwer
fixed-point theorem and its consequence: the fact that a~map $S \to S$ extends
to a~map $Q \to S$ if and only if its topological degree is zero. Given
a~$\vdim$-dimensional cell~$\omega$ of~$E$, we need to ensure that $g|_{\cBdry
  \omega}$ has topological degree zero. Recalling that $\xi(\delta_{\vdim}
\omega) = \zeta \circ p (\delta_{\vdim} \omega) = 0$ whenever $\omega \in
C_{\vdim}$, the required condition on~$g$ follows.

To conclude the argument, we observe that the~$2^{N}$-multiplication of~$(X,Q)$ is
the same as the $2$-multiplication of~$(W,Q)$, where~$W$ is the
$2^{N-1}$-multiplication of~$(X,Q)$; thus, we get the result by induction.

\begin{definition}
    For $k \in \nat$ we set $\sphere{k} = \R^{k+1} \cap \Bdry{\cball 01}$.
\end{definition}

\begin{definition}[\protect{cf.~\cite[Chap.~0, p.~14]{Hatcher2002}}]
    \label{def:hep}
    Let $X$ be a topological space and $A \subseteq X$ be a subspace. Set $I =
    [0,1] \subseteq \R$. We say that the pair $(X,A)$ has the~\emph{homotopy
      extension property HEP} if for every topological space~$Y$ every
    continuous function $h : (X \times \{0\}) \cup (A \times I) \to Y$ extends to
    a~continuous homotopy $H : X \times I \to Y$.
\end{definition}

\begin{remark}[\protect{cf.~\cite[Chap.~0, Example~0.15, p.~15]{Hatcher2002}}]
    \label{rem:hep-suff}
    If $k \in \natp$, $A \subseteq X \subseteq \R^{\adim}$, $A$ is compact of
    dimension~$k$, and there exists an~open set~$U \subseteq \R^{\adim}$ such
    that~$A \subseteq U \subseteq X$ and~$U$ is homeomorphic to~$A \times
    \R^{\adim - k}$ (i.e. $U$ is a~trivial vector bundle over~$A$ with
    fiber~$\R^{\adim-k}$), then $(X,A)$ has the HEP. In~particular, if~$A$ is
    a~sum of a~finite set of~$k$-dimensional cubes and $A \subseteq \Int X$,
    then $(X,A)$ has the HEP.
\end{remark}

\begin{remark}[\protect{cf.~\cite[Chap.~0, Prop.~0.17, p.~15]{Hatcher2002}}]
    \label{rem:hep-htp-equiv}
    If $(X,A)$ has the HEP and $A$ is contractible, then $X$ and $X/A$ are
    homotopy equivalent.
\end{remark}
  \begin{remark}
      We shall also use the following simple facts:
      \begin{itemize}
      \item if $X,Y \subseteq \R^{\adim}$, $A = X \cap Y$, and both~$(X,A)$
          and~$(Y,A)$ have the HEP, than $(X \cup Y, A)$ has the HEP;
      \item if $(X,A)$ has the HEP and $X \subseteq Y$, then $(Y,A)$ has the HEP.
      \end{itemize}
  \end{remark}

\begin{lemma}
    \label{lem:thick-retr}
    Assume $S,X \subseteq \R^{\adim}$ are compact, $S \subseteq X$, $\varepsilon
    \in (0,1)$, $(Y,S)$ has the HEP for any $Y \subseteq
      \R^{\adim}$ with $S \subseteq \Int Y$, and there exists a~Lipschitz
    retraction $\pi : S + \cball{0}{\varepsilon} \to S$. Let $j : S \to
    \R^{\adim}$ be the inclusion map.
    
    The following properties are equivalent:
    \begin{enumerate}
    \item
        \label{i:retr:lip}
        $S$ is a Lipschitz retract of $X$;
    \item
        \label{i:retr:cont}
        $S$ is a retract of $X$;
    \item
        \label{i:retr:thick-cont}
        there exists $\delta \in (0,\varepsilon)$ such that $S$ is a retract of $X +
        \cball{0}{\delta}$;
    \item
        \label{i:retr:deg-one}
        there exist a~continuous map $f : X \to S$ such that $\deg(f \circ j) = 1$.
    \end{enumerate}
\end{lemma}

\begin{proof}
    Clearly the implications $\ref{i:retr:lip} \Rightarrow \ref{i:retr:cont}$,
    $\ref{i:retr:thick-cont} \Rightarrow \ref{i:retr:cont}$, $\ref{i:retr:cont}
    \Rightarrow \ref{i:retr:deg-one}$ hold.

    \emph{Proof of $\ref{i:retr:cont} \Rightarrow \ref{i:retr:lip}$:} Assume $r
    : X \to S$ is a~retraction. Using the Tietze extension theorem (see
    e.g.~\cite[Chap.~7, Problem~O, p.~242]{Kelley1975}), we extend~$r$ to
    a~continuous function $\tilde R : \R^{\adim} \to \R^{\adim}$. We
    mollify~$\tilde R$ to obtain a~smooth function $R : \R^{\adim} \to
    \R^{\adim}$ such that $|R(x) - r(x)| \le 2^{-12} \varepsilon$ for $x \in X$;
    in~particular, $\dist(R(x),S) \le 2^{-12} \varepsilon$ for $x \in X$ so $\pi
    \circ R : X \to S$ is well defined. Since $r(x) = \pi(x)$ for $x \in S$,
    there exists $\delta \in (0,\varepsilon)$ such that $|R(x) - \pi(x)| \le
    2^{-8}\varepsilon$ for $x \in S + \cball{0}{\delta}$. Finally, we define a
    Lipschitz retraction $f : X \to S$ by
    \begin{displaymath}
        f(x) = \left\{
            \begin{aligned}
              &\pi(x) &&\text{if $\dist(x,S) \le 2^{-8}\delta$} \,,
              \\
              &\pi(R(x)) &&\text{if $\dist(x,S) \ge 2^{-7}\delta$} \,,
              \\
              &\pi\bigl( (1-t) \pi(x) + t \pi(R(x)) \bigr) &&\text{if $t = 2^8\dist(x,S)/\delta-1 \in (0,1)$} \,.
            \end{aligned}
        \right.
    \end{displaymath}

    \emph{Proof of $\ref{i:retr:cont} \Rightarrow \ref{i:retr:thick-cont}$:}
    Assume $r : X \to S$ is a retraction. Once again we extend $r$ to a continuous
    function $R : \R^{\adim} \to \R^{\adim}$. Note that $R$ is uniformly
    continuous on every compact subset of~$\R^{\adim}$; hence, there exists
    $\delta \in (0,1)$ such that $R \lIm X + \cball{0}{\delta} \rIm \subseteq S
    + \cball{0}{\varepsilon}$. We get that $\pi \circ R|_{X + \cball{0}{\delta}}$
    is the desired retraction.

    \emph{Proof of $\ref{i:retr:deg-one} \Rightarrow \ref{i:retr:cont}$:} Let $f
    : X \to S$ be continuous and such that $\deg(f \circ j) = 1$. Then there
    exists a~continuous homotopy $h : S \times I \to S$ such that $h(x,0) =
    f(x)$ and $h(x,1) = x$ for $x \in S$. We extend~$f$ to a continuous function
    $F : \R^{\adim} \to \R^{\adim}$ using the Tietze extension theorem and we
    find $\delta \in (0,1)$ such that $F \lIm X + \cball{0}{\delta} \rIm
    \subseteq S + \cball{0}{\varepsilon}$. Set $Y = X + \cball{0}{\delta}$.
    Observe that $\pi \circ F|_Y : Y \to S$ is well defined. Recall that $(Y,S)$
    has the HEP so we may extend~$h$ to a~homotopy $H : Y \times I \to S$ such
    that $H(x,0) = \pi(F(x))$ for every $x \in Y$.  The~desired retraction $r :
    X \to S$ is then given by~$r(x) = H(x,1)$ for~$x \in X$.
\end{proof}

\begin{definition}
    \label{def:wedge-sum}
    Assume $J$ is an index set and for each $\alpha \in J$ we are given
    a~pointed topological space $(X_{\alpha},x_{\alpha})$. We define the
    \emph{wedge sum} to be the pointed topological space
    \begin{displaymath}
        {\textstyle \bigvee_{\alpha \in J} (X_{\alpha},x_{\alpha})
          = \bigl( \bigcup \bigl\{ X_{\alpha} \times \{\alpha\} : \alpha \in J \bigr\} \bigr)
          / \bigl\{ (x_{\alpha},\alpha) : \alpha \in J \bigr\} } 
    \end{displaymath}
    endowed with the quotient topology.

    If $J = \{ 1,2,\ldots N \}$ for some $N \in \natp$, then we use the notation
    \begin{displaymath}
        {\textstyle \bigvee_{\alpha \in J} (X_{\alpha},x_{\alpha})}
        = {\textstyle \bigvee_{i=1}^N (X_i,x_i)}
        = (X_1,x_1) \vee (X_2,x_2) \vee \cdots \vee (X_N,x_N) \,.
    \end{displaymath}
\end{definition}

\begin{remark}
    \label{rem:wedge-proj}
    \begin{enumerate}
    \item Let $Z = \bigvee_{\alpha \in J} (X_{\alpha},x_{\alpha})$ and $\alpha
        \in J$. There exist continuous maps $\tau_{\alpha} : X_{\alpha}
        \hookrightarrow Z$ and $\pi_{\alpha} : Z \twoheadrightarrow
        X_{\alpha}$. The first one is simply the inclusion and the second comes
        from the projection $Z \twoheadrightarrow Z / \bigvee_{\beta \in J
          \without \{\alpha\}} (X_{\beta},x_{\beta})$.
    \item For each $\alpha \in J$ assume $(X_{\alpha},x_{\alpha})$ and
        $(Y_{\alpha},y_{\alpha})$ are pointed topological spaces and there exist
        maps $f_{\alpha} : (X_{\alpha},x_{\alpha}) \to (Y_{\alpha},y_{\alpha})$
        and $g_{\alpha} : (Y_{\alpha},y_{\alpha}) \to (X_{\alpha},x_{\alpha})$
        such that $f_{\alpha} \circ g_{\alpha} \approx \id{Y_{\alpha}}$ and
          $g_{\alpha} \circ f_{\alpha} \approx \id{X_{\alpha}}$. Then
            $\bigvee_{\alpha \in J} (X_{\alpha},x_{\alpha})$ and
            $\bigvee_{\alpha \in J} (Y_{\alpha},y_{\alpha})$ are homotopy
            equivalent.
    \end{enumerate}
\end{remark}

\begin{definition}[\protect{cf.~\cite[\S{3}]{Fomenko1986}}]
    \label{def:cw-complex}
    A \emph{CW-complex} is a topological space $X$ such that for $l \in \nat$
    there exist: an index set $J_l$, a family of $l$-dimensional balls $\{
    \sigma^l_i : i \in J_l \}$, and for each $i \in J_l$ there is a continuous
    \emph{characteristic map} $\varphi^l_i : \sigma^l_i \to X$ such that
    \begin{enumerate}
    \item setting $X^{-1} = \varnothing$ and $X^k = \bigcup_{l=0}^{k} \bigcup_{i
          \in J_l} \im \varphi^l_i$ for $k \in \nat$, we have $X =
        \bigcup_{k=0}^{\infty} X^k$;
    \item $\varphi^l_i$ restricted to $\Int \sigma^l_i$ is a~homeomorphic
        embedding;
    \item the image of $\Bdry{\sigma^l_i}$ under $\varphi^l_i$ is contained
        in~$X^{l-1}$;
    \item
        \label{i:cw:finite-cells}
        the image of $\varphi^l_i$ intersects only finitely many images of
        other characteristic maps;
    \item a set $F \subseteq X$ is closed in~$X$ if and only if
        $(\varphi^l_i)^{-1} \lIm F \rIm$ is closed in $\sigma^l_i$ for all $l
        \in \nat$ and $i \in J_l$.
    \end{enumerate}

    The image of any $\varphi^l_i$ shall be called an \emph{$l$-dimensional
      cell} of~$X$ and the set $X^l$ the \emph{$l$-skeleton} of~$X$. If $X =
    X^k$ for some $k \in \nat$, then we say that $X$ is \emph{$k$-dimensional}
    and if, in addition, all the sets $J_l$ for $l \in \{0,1,\ldots,k\}$ are
    finite, then we say that $X$ is a~\emph{finite CW-complex}.
\end{definition}

\begin{remark}
    A CW-complex $X$ can also be seen as constructed inductively by attaching
    cells $\sigma^l_i$ to $X^{l-1}$ via maps $\varphi^l_i|_{\Bdry{\sigma^l_i}}$;
    cf.~\cite[Chap.~0, p.~5]{Hatcher2002}.
\end{remark}

\begin{remark}
    \label{rem:cubes-cw}
    If $\mathcal A \subseteq \cubes_*^{\adim}$, then $X = \bigcup \mathcal A$ is
    a~CW-complex with $X^k = \tbcup \{ Q \in \cubes^{\adim}_{k} : Q \subseteq X
    \}$ for $k \in \{ 0, 1, \ldots, \adim \}$. If $\mathcal A$ is finite, then
    $X$ is a finite CW-complex.
\end{remark}

\begin{remark}
    \label{rem:cw-homology}
    Assume $X$ is a CW-complex. We shall use cellular homology of~$X$;
    see~\cite[\S{12}]{Fomenko1986} or~\cite[\S{2.2}, p.~137]{Hatcher2002}.
    Recall that for $l \in \nat$ the chain group
    \begin{displaymath}
        C_l(X) = \Sh_l(X^{l},X^{l-1})
    \end{displaymath}
    is the free abelian group with basis $\{ \sigma^l_i : i \in J_l \}$. Next,
    define the differentials
    \begin{gather}
        d_0 : C_0 \to \{ 0 \}
        \quad \text{and} \quad
        d_l : C_l(X) \to C_{l-1}(X)
        \\
        \label{eq:dl}
        \text{by} \quad
        d_l(\sigma^l_i) = \sum_{j \in J_{l-1}} \deg(\psi^l_{i,j}) \sigma^{l-1}_{j} 
        \quad \text{for $l \in \natp$}
        \,,
    \end{gather}
    where $\psi^l_{i,j}$ is defined as the composition
    \begin{equation}
        \label{eq:dl-coeff}
        \Bdry{\sigma^l_{i}} 
        \xrightarrow{\varphi^l_i|_{\Bdry{\sigma^l_i}}} X^{l} 
        \twoheadrightarrow X^l/(X^l \without \sigma^{l-1}_j)
        \xrightarrow{\simeq} \sphere{l-1} \,.
    \end{equation}
    Clearly, by~\ref{def:cw-complex}\ref{i:cw:finite-cells}, the sum
    in~\eqref{eq:dl} is finite. Moreover, $(C_l(X),d_l)_{l=0}^{\infty}$ defines
    a~chain-complex whose homology groups coincide with singular homology groups
    of~$X$; see~\cite[Theorem~2.35]{Hatcher2002} or~\cite[\S{12},
    p.~94]{Fomenko1986}.
\end{remark}

\begin{remark}
    \label{rem:free-groups}
    Let $F$ be a free abelian group. The following observations shall become
    particularly useful:
    \begin{enumerate}
    \item
        \label{i:free:basis}
        If~$G$ is a subgroup of~$F$, then $G$ is itself a free abelian group;
        cf.~\cite[I,\S{}7,Theorem~7.3]{Lang2002}.
    \item
        \label{i:free:splitting}
        If $G$ is another free abelian group and $d : F \to G$, then $F$ splits
        into a direct sum $F = \ker d \oplus H$ for some subgroup~$H$ of~$F$.

        To prove the above claim (b), let $A = \im d \subseteq G$. Then $A$ is
        a~subgroup of~$G$; hence, $A$ is a free abelian group. Let $\{ a_i : i
        \in J \}$ be a basis of~$A$. In order to prove the existence of a splitting, it suffices
        to define a~homomorphism $f : A \to F$ such that $d \circ f =
        \id{A}$. For each $i \in J$ we choose arbitrarily $b_i \in F$ such that
        $d(b_i) = a_i$ and set $f(a_i) = b_i$. Then $f$ extends to
        a~homomorphism $A \to F$ simply because~$A$ is free.
    \end{enumerate}
\end{remark}

Next, we prove that if $X$ is a $(k+1)$-dimensional CW-complex, then any
homomorphism from the $k^{\mathrm{th}}$~homology group~$\Sh_k(X)$ to the group
of integers $\integers$ is induced by some map $X \to \sphere{k}$.

\begin{lemma}
    \label{lem:hom-real}
    Assume $k \in \nat$, $X$ is a $(k+1)$-dimensional CW-complex, and there is
    given a homomorphism $\zeta : \Sh_k(X) \to \integers$. Then there exists $f
    : X \to \sphere{k}$ such that $f_* = \zeta$.
\end{lemma}

\begin{proof}
    For $l \in \{0,1,2,\ldots,k+1\}$ let $J_l$ be the set indexing
    $l$-dimensional cells of~$X$ and for $i \in J_l$ let $\{ \sigma^l_i : i \in
    J_l \}$, $\varphi^l_i : \sigma^l_i \to X$, $d_l$, $C_l(X)$, $X^l$ be defined
    as in~\ref{def:cw-complex} and~\ref{rem:cw-homology}.

    By definition $C_k(X)$ are free abelian groups. Set $K = \ker d_{k}
    \subseteq C_k(X)$ and employ~\ref{rem:free-groups}\ref{i:free:splitting} to
    find another subgroup $L \subseteq C_k(X)$ such that $C_k(X) = K \oplus L$.
    Let $p : K \twoheadrightarrow \Sh_k(X)$ and $q : K \oplus L
    \twoheadrightarrow K$ be canonical projections. Define $\xi : C_k(X) \to
    \integers$ as the composition
    \begin{displaymath}
        C_k(X) \xrightarrow{\ q\ } K
        \xrightarrow{\ p\ } \Sh_k(X)
        \xrightarrow{\ \zeta \ } \integers \,.
    \end{displaymath}
    We~record now some trivial observations
    \begin{equation}
        \label{eq:im-dk+1}
        \zeta(x) = 0 \quad \text{whenever $x \in \Sh_k(X)$ has finite order} \,,
        \quad
        \zeta \circ p = \xi|_K \,,
        \quad
        \xi \circ d_{k+1} = 0 \,.
    \end{equation}
    We~shall first construct $\gamma : X^k \to \sphere{k}$ such that $\gamma_* :
    \Sh_k(X^k) \to \integers$ equals $\zeta \circ p$ and then extend $\gamma$
    to~$f : X^{k+1} \to \sphere{k}$ using a bit of obstruction theory.

    For each $i \in J_k$ the space $\sigma^k_i / \Bdry{\sigma^k_i}$ is
    homeomorphic to $\sphere{k}$ and we define
    \begin{displaymath}
        \gamma_i : \sigma^k_i / \Bdry{\sigma^k_i} \to \sphere{k}
        \quad \text{so that} \quad
        \deg(\gamma_i) = \xi(\sigma^k_i) \,.
    \end{displaymath}
    Note that the space $X^k/X^{k-1}$ is homeomorphic to the wedge sum
    of~topological spheres $\bigvee_{i \in J_k} (\sigma^k_i / \Bdry{\sigma^k_i},
    \eqclsl \Bdry{\sigma^k_i} \eqclsr)$.  We~construct the map
    \begin{displaymath}
        \tilde \gamma : X^k/X^{k-1} \to \sphere{k} 
        \quad \text{so that} \quad
        \tilde \gamma|_{\sigma^k_i / \Bdry{\sigma^k_i}} = \gamma_i
        \quad \text{for $i \in J_k$} \,.
    \end{displaymath}
    Let $\pi : X^k \twoheadrightarrow X^k/X^{k-1}$ be the projection. Finally,
    set
    \begin{displaymath}
        \gamma = \tilde \gamma \circ \pi \,.
    \end{displaymath}
    Note that $\Sh_k(X^k) = K$. One readily verifies that $\gamma_* = \xi|_K =
    \zeta \circ p$.

    Now we need to extend $\gamma$ to the $(k+1)$-dimensional cells in~$X$.
    Employing the obstruction theory~\cite[\S{17}]{Fomenko1986} this is possible
    if for each $j \in J_{k+1}$ the composition
    \begin{displaymath}
        \Bdry{\sigma^{k+1}_j} \xrightarrow{\varphi^{k+1}_j|_{\Bdry{\sigma^{k+1}_j}}}
        X^k \xrightarrow{\ \gamma\ } \sphere{k}
    \end{displaymath}
    has topological degree zero. However, this degree equals exactly
    $\xi(d_{k+1}(\sigma^{k+1}_{j}))$ which is zero by~\eqref{eq:im-dk+1}.
    Therefore, there exists $f : X \to \sphere{k}$ such that $f|_{X^k} =
    \gamma$; in~particular, $f_* : \Sh_k(X) \to \integers$ equals~$\zeta$.
\end{proof}

\begin{remark}
    \label{rem:alt-proof}
    Employing some more sophisticated tools of algebraic topology, a shorter
    proof of Lemma \ref{lem:hom-real} can be given as follows. The universal coefficient
    theorem~\cite[Theorem~3.2]{Hatcher2002} provides an epimorphism
    \begin{displaymath}
        h : \Sh^{k}(X;\integers) \twoheadrightarrow \Hom(\Sh_k(X),\integers) \,.
    \end{displaymath}
    On the other hand, there exists an isomorphism
    (see~\cite[Theorem~4.57]{Hatcher2002})
    \begin{displaymath}
        T : \bigl[ X, K(\integers,k) \bigr]_{\mathrm{htp}} \xrightarrow{\ \simeq\ } \Sh^{k}(X;\integers) \,,
    \end{displaymath}
    where $[ X, K(\integers,k) ]_{\mathrm{htp}}$ denotes the set of homotopy
    classes of maps $X \to K(\integers,k)$ and $K(\integers,k)$ is the
    Eilenberg-MacLane space; cf.~\cite[\S{4.2}, p.~365]{Hatcher2002}. Therefore,
    any homomorphism $\Sh_k(X) \to \integers$ is induced by some map $X \to
    K(\integers,k)$. Observing, that $K(\integers,k)$ is a CW-complex obtained
    from the sphere $\sphere{k}$ by gluing in cells of dimension at least $k+2$,
    we see, since $X$ is $(k+1)$-dimensional and the homotopy groups
    $\pi_l(\sphere{k+2}) = 0$ for $l \in \{1,2,\ldots,k+1\}$, that any map $X
    \to K(\integers,k)$ is homotopic to a map whose image lies in $\sphere{k}$.
\end{remark}

\begin{remark}
    \label{rem:cp2}
    The bound on the dimension of~$X$ plays a crucial role
    in~\ref{lem:hom-real}. Indeed, if the dimension of $X$ is bigger than~$k+1$,
    then an element of~$\Hom(\Sh_k(X),\integers)$ might not be induced by a~map
    $X \to \sphere{k}$ as the following example shows. Let $k = 2$ and $X$ be
    the complex projective space of real dimension~$4$ (often denoted
    $\mathbf{CP}^2$). Then $X$ is a~CW-complex constructed by attaching
    a~$4$-dimensional cell to~$\sphere{2}$ via the Hopf fibration $\sphere{3}
    \to \sphere{2}$. We have
    \begin{displaymath}
        \Sh_2(X) =\Sh^2(X) = \Sh^4(X) = \integers \,.
    \end{displaymath}
    Recall that $\Sh^*(X)$ is the graded ring $\integers[\sigma] / \sigma^3$,
    where $\sigma$ is the generator of~$\Sh^2(X)$; 
      cf.~\cite[Theorem~3.12]{Hatcher2002}.  Finally, since all the homology
    and cohomology groups of~$X$ are free, the universal coefficient theorem
    provides a natural isomorphism
    \begin{displaymath}
        j : \Sh^{2}(X) \xrightarrow{\simeq} \Hom(\Sh_2(X),\integers) \,.
    \end{displaymath}

    Assume there exists a map $f : X \to \sphere{2}$ such that $f_* : \Sh_2(X)
    \to \Sh_2(\sphere{2})$ is an isomorphism. In~consequence, $f^* :
    \Sh^2(\sphere{2}) \to \Sh^2(X)$ is also an isomorphism. However, the map
    $f^*$ is a~\emph{homomorphism of graded rings} and this gives a
    contradiction because the square of the generator of~$\Sh^2(\sphere{2})$ is
    zero while the square of the generator of~$\Sh^2(X)$ is the generator
    of~$\Sh^4(X)$.
\end{remark}

\begin{corollary}
    \label{cor:deg-gcd}
    Let $k \in \nat$, $X$ be a $(k+1)$-dimensional CW-complex, and $j :
    \sphere{k} \to X$ be continuous. Define
    \begin{displaymath}
        D = \bigl\{ | \deg(f \circ j) | : f : X \to \sphere{k} \text{ continuous} \bigr\} \without \{0\} \,.
    \end{displaymath}
    If $D \ne \varnothing$ and $A = \min D$, then
    \begin{displaymath}
        D = \{ m A : m \in \natp  \} \,.
    \end{displaymath}
\end{corollary}

\begin{proof}
    If $D = \varnothing$ there is nothing to prove, so we assume $D \ne
    \varnothing$. Let $f_1,f_2 : X \to \sphere{k}$ be two continuous maps such
    that $d_i = | \deg(f_i \circ j) | \in \natp$ for $i \in \{1,2\}$. Set $d =
    \gcd(d_1,d_2)\in \natp$. By the Euclidean algorithm, there exist integers
    $c_1, c_2$ such that $d = c_1 d_1 + c_2 d_2$. We employ~\ref{lem:hom-real} to
    find a map $f : X \to \sphere{k}$ such that $f_* = c_1 f_{1*} + c_2
    f_{2*}$. Then $| \deg(f \circ j) | = d \in D$.

    We have shown that whenever $d_1,d_2 \in D \subseteq \natp$, then
    $\gcd(d_1,d_2) \in D$. Moreover, if $f : X \to \sphere{k}$, $|\deg(f \circ
    j)| = A \in D$, and $m \in \natp$, then $mA \in D$ because one can
    post-compose $f$ with a map $\sphere{k} \to \sphere{k}$ of degree~$m$.
\end{proof}

\begin{corollary}
    \label{cor:deg-wedge-sum}
    Let $k,N \in \nat$, $X$ be $(k+1)$-dimensional CW-complex, $x_0 \in X$, $Z =
    \bigvee_{i=1}^{N} (X,x_0)$ and $j : \sphere{k} \to Z$ be continuous. For $l
    \in \{1,2,\ldots,N\}$ define $\pi_l : Z \to X$ as in~\ref{rem:wedge-proj}.
    Assume there exists $\varphi : \sphere{k} \to X$ such that for $l \in
    \{1,2,\ldots,N\}$ the map $\pi_l \circ j : \sphere{k} \to X$ is homotopic
    either to~$\varphi$ or to the constant map and $\pi_1 \circ j \approx
    \varphi$. Set
    \begin{gather}
        D = \bigl\{ | \deg(f \circ j) | : f : Z \to \sphere{k} \text{ continuous} \bigr\} \,,
        \\
        E = \bigl\{ | \deg(g \circ \varphi) | : g : X \to \sphere{k} \text{ continuous} \bigr\} \,.
    \end{gather}
    Then $D = E$.
\end{corollary}

\begin{proof}
    For $l \in \{1,2,\ldots,N\}$ let $\tau_l : X \to Z$ be the injection as
    in~\ref{rem:wedge-proj}. If $g : X \to \sphere{k}$ is continuous, then $f =
    g \circ \pi_1 : Z \to \sphere{k}$ is homotopic to $g \circ \varphi$ so
    $\deg(g \circ \varphi) = \deg(f \circ j)$ and we get $E \subseteq D$.
    On the other hand if $f : Z \to \sphere{k}$, then we consider the maps $f_l
    = f \circ \tau_l : X \to \sphere{k}$ for $l \in \{1,2,\ldots,N\}$ to see
    that
    \begin{displaymath}
        D \ni | \deg(f \circ j) | =  \bigl| {\textstyle \sum_{l=1}^N \deg(f_l \circ \pi_l \circ j)} \bigr| \in E 
        \quad \text{by~\ref{cor:deg-gcd}}\,;
    \end{displaymath}
    thus, $D \subseteq E$.
\end{proof}

\begin{lemma}
    \label{lem:cw-structure}
    Let $J = [0,2]$, $\varepsilon \in (0,\infty)$ and assume
    \begin{gather}
        Q \in \cubes^{\adim}_{\vdim}\,,
        \quad
        S = \cBdry{Q} \,,
        \quad
        X \subseteq \R^{\adim} \text{ is compact} \,,
        \quad
        S \subseteq X \,,
        \quad
        \HM^{\vdim}(X) < \infty \,.
    \end{gather}
    Then there exist: a Lipschitz map $f : I \times \R^{\adim} \to \R^{\adim}$,
    a~compact set $E \subseteq \R^{\adim}$, an open set~$U \subseteq
    \R^{\adim}$, and a~finite set $\mathcal B \subseteq \cubes^{\adim}_{\vdim}$
    such that
    \begin{gather}
        S \subseteq E = \tbcup \mathcal B = f\lIm \{2\} \times U \rIm \,,
        \quad
        X \subseteq U \subseteq X + \cball{0}{\varepsilon} \,,
        \quad
        f \lIm J \times U \rIm \subseteq V \,,
        \\
        f(t,x) = x \quad \text{for $(t,x) \in I \times E$}  \,,
        \quad
        \text{$E$ is a strong deformation retract of $U$} \,.
    \end{gather}
\end{lemma}

\begin{proof}
    For $R \in \cubes^{\adim}$ denote by $\tilde R$ the $\adim$-dimensional cube
    with the same center as~$R$ and side-length three times bigger than~$R$.
    Let $N \in \natp$ be such that $2^{-N+4} \sqrt{\adim} < \min\{ \varepsilon,
    \side{Q} \}$ and define
    \begin{displaymath}
        \mathcal A = \bigl\{ R \in \cubes^{\adim}_{\adim}(N) : \tilde R \cap X \ne \varnothing \bigr\}  \,.
    \end{displaymath}
    Apply~\ref{thm:mae:dt} with $\cubes^{\adim}_{\adim}$, $\mathcal A$, $X$ in
    place of~$\mathcal F$, $\mathcal A$, $S$ to obtain a~Lipschitz map $f : J
    \times \R^{\adim} \to \R^{\adim}$, an~open set $V \subseteq \R^{\adim}$, and
    a~finite set $\mathcal B \subseteq \cubes^{\adim}_{\vdim}(N)$. Set $E =
    \tbcup \mathcal B$ and $U = V \cap \Int{\tbcup \mathcal A}$ and
    recall~\ref{rem:dt-retract}. Since $S \subseteq \tbcup
    \cubes^{\adim}_{\vdim-1}(N)$ we get $S \subseteq E$.
\end{proof}

For convenience and brevity of the notation we introduce the following
definition.

\begin{definition}
    \label{def:Rinfty}
    We define $\R^{\infty}$ to be the direct sum of countably many copies
    of~$\R$ and for $i \in \natp$ we let $e_i \in \R^{\infty}$ be the standard
    basis vector of the $i^{\mathrm{th}}$ copy of~$\R$. Thus, $\R^{\infty}$ is
    the set of all finite linear combinations of the vectors $\{ e_i : i \in
    \natp\}$.
\end{definition}

We want to compare, up to homotopy, a multiplication $(Y,Q)$ of some cubical
test pair $(X,Q)$ with the wedge sum of certain number of copies of~$X$.
However, it might happen that two copies of~$X$ placed side by side intersect
outside~$\cBdry{Q}$. To prevent this, we define a lifted multiplication so that
different copies of~$X$ intersect only along $\cBdry{Q}$.

\begin{definition}
    \label{def:lifted-mult}
    Let $X$, $Q$, $k$, $\mathcal A = \{ K_1, \ldots, K_{k^{\vdim}} \}$ be as
    in~\ref{def:multiplication}. Let $e_i$ for $i \in \natp$ be as
    in~\ref{def:Rinfty}. Define $j : \R^{\adim} \to \R^{\adim} \times
    \R^{\infty}$, $p : \R^{\adim} \times \R^{\infty} \to \R^{\adim}$, and
    $\eta_i : \R^{\adim} \to \R^{\adim} \times \R^{\infty}$ for $i \in \{1, 2,
    \ldots, k^{\vdim} \}$ by
    \begin{displaymath}
        j(x) = (x,0) \,,
        \quad
        p(x,y) = x \,,
        \quad 
        \eta_i(x) = j \circ \trans{\centre{K_i}} \circ \scale{1/k} \circ \trans{-\centre{Q}}(x) + \dist(x,\cBdry{Q}) e_i \,.
    \end{displaymath}
    We say that $(Y,j\lIm Q \rIm)$ is the \emph{lifted $k$-multiplication} of
    $(X,Q)$ if
    \begin{displaymath}
        Y = \tbcup \bigl\{ \eta_i \lIm X \rIm : i \in \{ 1,2, \ldots, k^{\vdim} \} \bigr\} \subseteq \R^{\adim} \times \R^{\infty} \,.
    \end{displaymath}
\end{definition}

\begin{lemma}
    \label{lem:ind-step}
    Assume
    \begin{gather}
        U \subseteq \R^{\adim} \text{ is open} \,,
        \quad
        Q = [0,1]^{\vdim} \times \{0\}^{\adim - \vdim} \in \cubes^{\adim}_{\vdim}(0) \,,
        \quad
        S = \cBdry{Q} \,,
        \quad
        N \in \natp \,,
        \\
        \mathcal B \subseteq \cubes^{\adim}_{\vdim} \text{ is finite} \,,
        \quad
        S \subseteq E = \tbcup \mathcal B \subseteq U \,,
        \quad
        \text{$E$ is a strong deformation retract of $U$} \,,
        \\
        \text{$j$ and $p$ are as in~\ref{def:lifted-mult}} \,,
        \quad
        \text{$(Y,j \lIm Q \rIm)$ is the lifted $2^N$-multiplication of $(U,Q)$} \,,
        \\
        \text{$(Z,j \lIm Q \rIm)$ is the lifted $2^{N-1}$-multiplication of $(U,Q)$} \,.
    \end{gather}
    If $j \lIm S \rIm$ is a Lipschitz retract of~$Y$, then $j \lIm S \rIm$ is
    a~Lipschitz retract of~$Z$.
\end{lemma}

\begin{proof}
    Suppose there exists a Lipschitz retraction $r : Y \to j \lIm S \rIm$.
    Due to~\ref{lem:thick-retr} it suffices to show that there exists
    a~continuous map $h : Z \to S$ such that $\deg(h \circ j|_S) = 1$.  Set $J =
    \{1,2,\ldots, 2^{\vdim}\}$. Let $(X,j \lIm Q \rIm)$ be the lifted
    $2^{N-1}$-multiplication of~$(E,Q)$ and $(F,j \lIm Q \rIm)$ be the lifted
    $2^{N}$-multiplication of~$(E,Q)$. Observe that $Y$ contains
    $2^{\vdim}$~copies of~$\scale{1/2} \lIm Z \rIm$; let us denote these copies
    $Z_1$, $Z_2$, \ldots, $Z_{2^{\vdim}}$ and the corresponding cubes $Q_1$,
    $Q_2$, \ldots, $Q_{2^{\vdim}}$ so that
    \begin{displaymath}
        Y = \tbcup \{ Z_i : i \in J \}
        \quad \text{and} \quad
        j \lIm Q \rIm = \tbcup \{ Q_i : i \in J \} \,.
    \end{displaymath}
    We also define
    \begin{displaymath}
        S_i = \cBdry{Q_i}
        \quad \text{and} \quad
        X_i = F \cap Z_i 
        \quad \text{for $i \in J$} \,.
    \end{displaymath}
    Let $T = \R^{\vdim} \times \{0\}^{\adim - \vdim} \in
    \grass{\adim}{\vdim}$. Then $Q \subseteq \vertex{Q} + T$. Let $(v_1, v_2,
    \ldots, v_{\adim})$ be the standard basis of~$\R^{\adim}$ and define
    \begin{gather}
        T_i = \lin \{ v_i \}^{\perp} \cap T \in \grass{\adim}{\vdim-1}
        \quad \text{for $i \in \{1,2,\ldots,\vdim\}$} \,,
        \\
        R = j \bigl\lIm \tbcup \bigl\{ (\centre{Q} + T_i) \cap Q : i \in \{1,2,\ldots,\vdim\} \bigr\} \bigr\rIm \subseteq Y \,.
    \end{gather}
    Note that $R$ and $R \cap Z_i$ for $i \in J$ are contractible. Since $U$ is
    open, we have $S \subseteq \Int U$ so the pairs $(Y,R)$ and $(Z_i,R \cap
    Z_i)$ for $i \in \{1,2,\ldots,\vdim\}$ all have the HEP
    by~\ref{rem:hep-suff}. Therefore, $R$ and $Y / R$ are homotopy equivalent
    by~\ref{rem:hep-htp-equiv}. Similarly, $Z_i$ and $Z_i / (R \cap Z_i)$ are
    homotopy equivalent for $i \in J$. Let $q_0 = j(\centre{Q})$. We shall write
    $\eqclsl q_0 \eqclsr$ for the equivalence class of~$q_0$ in a~given quotient
    space. Denoting homotopy equivalence by~``$\approx$'' and homeomorphism
    by~``$\simeq$'' we obtain
    \begin{displaymath}
        Y \approx Y / R
        \simeq {\textstyle \bigvee_{i=1}^{2^{\vdim}}} ( Z_i/(Z_i \cap R), \eqclsl q_0 \eqclsr )
        \approx {\textstyle \bigvee_{i=1}^{2^{\vdim}}} ( Z_i, q_0 ) \,.
    \end{displaymath}
    Set
    \begin{displaymath}
        W = {\textstyle \bigvee_{i=1}^{2^{\vdim}}} ( Z_i, q_0 ) \,,
        \quad
        M = {\textstyle \bigvee_{i=1}^{2^{\vdim}}} ( X_i, q_0 ) \,,
        \quad
        \Sigma = {\textstyle \bigvee_{i=1}^{2^{\vdim}}}(S_i,q_0) \,,
    \end{displaymath}
    and note that $\Sigma \subseteq M \subseteq W$.  Let $\varphi : Y \to W$ and
    $\psi : W \to Y$ be such that $\varphi \circ \psi \approx \id{W}$ and $\psi
    \circ \varphi \approx \id{Y}$. For $i \in J$ let $\pi_i : \Sigma \to S_i$ be
    the projection defined in~\ref{rem:wedge-proj}.  Observe that
    \begin{displaymath}
        \varphi \circ j \lIm S \rIm = \Sigma 
        \quad \text{and} \quad
        \deg(\pi_i \circ \varphi \circ j|_S) = 1
        \quad \text{for $i \in J$} \,.
    \end{displaymath}
    Recall that $E$ is a~strong deformation retract of~$U$; hence, if~$\xi : M
    \hookrightarrow W$ is the inclusion map, there exists a~continuous maps
    $\zeta : W \to M$ such that $\xi \circ \zeta \approx \id{W}$ and $\zeta
    \circ \xi \approx \id{M}$. Moreover, $\xi|_{\Sigma} = \zeta|_{\Sigma} =
    \id{\Sigma}$. Since $E = \tbcup \mathcal B$ we see that $E$ and~$M$ are
    $\vdim$-dimensional CW-complexes by~\ref{rem:cubes-cw}.  Hence, we may
    apply~\ref{cor:deg-wedge-sum} to deduce that
    \begin{multline}
        \bigl\{ |\deg(f \circ \zeta \circ \varphi \circ j|_S)| : f : M \to S \text{ continuous} \bigr\}
        \\
        = \bigl\{ |\deg(g|_S)| : g : X \to S \text{ continuous} \bigr\} \,.
    \end{multline}
    However, if we take $f = p \circ r \circ \psi \circ \xi : M \to S$, then
    \begin{displaymath}
        f \circ \zeta \circ \varphi \circ j|_S
        = p \circ r \circ \psi \circ \xi \circ \zeta \circ \varphi \circ j|_S 
        \approx p \circ r \circ j|_S = \id{S} \,.
    \end{displaymath}
    Therefore, there exists $g : X \to S$ such that $\deg(g \circ j|_S) = 1$.
    Let $\alpha : X_1 \to X$ and $\beta : Z \to Z_1$ be homeomorphisms composed
    of homotheties and translations. Then, recalling~$\zeta|_{\Sigma} =
    \id{\Sigma}$, the composition
    \begin{displaymath}
        S \xrightarrow{\ j|_S\ } Z \xrightarrow{\ \beta\ } Z_1 \xrightarrow{\ \zeta|_{Z_1}\ } X_1 \xrightarrow{\ \alpha\ } X \xrightarrow{\ g\ } S
    \end{displaymath}
    equals $g \circ j|_S$ and has degree one. Employing~\ref{lem:thick-retr} we
    obtain a Lipschitz retraction $Z \to S$.
\end{proof}

\begin{corollary}
    \label{cor:mult-down}
    If $S$ and $U$ are as in~\ref{lem:ind-step}, then $S$ is a Lipschitz retract of~$U$.
\end{corollary}

\begin{proof}
    We assume $j \lIm S \rIm$ is a Lipschitz retract of~$Y$, where $Y$ is the
    lifted $2^N$-mul\-ti\-pli\-cation of $(U,Q)$. We proceed by induction with
    respect to $N \in \nat$. If $N = 0$, we have $j \lIm U \rIm = Y$ so $S$ is
    a~Lipschitz retract of~$U$ by assumption. The inductive step is now a~direct
    application of~\ref{lem:ind-step}.
\end{proof}

\begin{theorem}
    \label{thm:all-tp-good}
    Assume $N \in \natp$, $(X,Q)$ is a cubical test pair, and $(Y,Q)$ is the
    $2^N$-mul\-ti\-pli\-ca\-tion of $(X,Q)$. Then $(Y,Q)$ is a cubical test pair.
\end{theorem}

\begin{proof}
    Using homotheties and rotations we may and shall assume that $Q =
    [0,1]^{\vdim} \times \{0\}^{\adim-\vdim} \in \cubes^{\adim}_{\vdim}(0)$.
    We~only need to show that $S = \cBdry{Q}$ is not a Lipschitz retract
    of~$Y$. Let $p$ and $j$ be as in~\ref{def:lifted-mult}. Assume, by
    contradiction, that there is a~Lipschitz retraction of~$Y$
    onto~$S$. Employing~\ref{lem:thick-retr} we find $\delta \in (0,1)$ such
    that~$S$ is a~retract of $Y +
    \cball{0}{2^{-N}\delta}$. Apply~\ref{lem:cw-structure} with $X$, $Q$,
    $\delta$ in place of $X$, $Q$, $\varepsilon$ to obtain a finite set
    $\mathcal B \subseteq \cubes^{\adim}_{\vdim}$ and an open set $U \subseteq X
    + \cball{0}{\delta}$ such that $E = \tbcup \mathcal B$ is a strong
    deformation retract of~$U$ and $X \subseteq U$. Let $(Z,j \lIm Q \rIm)$ be
    the lifted $2^N$-mul\-ti\-pli\-cation of $(U,Q)$. Clearly $p \lIm Z \rIm = Y$ and
    $p \circ j|_S = \id{S}$, so $j \lIm S \rIm$ is a Lipschitz retract of~$Z$.
    Applying~\ref{lem:ind-step} to $U$, $Q$, $N$, $\mathcal B$
    and then~\ref{cor:mult-down}, we conclude that~$S$ is
    a~Lipschitz retract of~$U$ which contains~$X$, so $S$ is also a~Lipschitz
    retract of~$X$ and this contradicts the assumption that $(X,Q)$ is a cubical
    test pair.
\end{proof}

\begin{remark}
    To conclude we gather all our results in one place. Let $x \in \R^{\adim}$,
    $\mathcal C$ be the set of all cubical test pairs, $\mathcal P$ be the set
    of all test pairs, $\mathcal R$ be the set of all rectifiable test
    pairs. Then
    \begin{enumerate}
    \item if $U \subseteq \R^{\adim}$ is open, $F \in \FwBC_x$ for all $x \in
        U$, $F$ is bounded, and $\mathcal G$ is a good class in the sense
        of~\cite[3.4]{FangKol2017}, then there exists $S \in \mathcal G$ such
        that $\Phi_F(S) = \inf\{ \Phi_F(R) : R \in \mathcal G \}$;
    \item $\FAE_x(\mathcal P) = \FAE_x(\mathcal C) = \FAE_x(\mathcal R)$
        and $\FAUE_x(\mathcal P) = \FAUE_x(\mathcal C) = \FAUE_x(\mathcal R)$;
    \item $\FAC_x = \FwBC_x \subseteq \FAE_x(\mathcal C)$.
    \end{enumerate}
    Moreover, if $\adim = \vdim+1$, then by~\cite[Theorem~1.3]{DDG2016rect}
    we~know that $F \in \FAC_x$ if and only if the function
    \begin{equation}
        \label{eq:F-to-G}
        G(x,\nu) = |\nu| F(x, \lin \{ \nu \}^{\perp}) 
        \quad \text{for every $x,\nu \in \R^{\adim}$}
    \end{equation}
    is strictly convex in all but the radial directions, namely
    \begin{displaymath}
        G(x,\nu) > \langle D_\nu G(x,\bar \nu),\nu\rangle
        \qquad \text{for every $x \in \R^{\adim}$, $\bar \nu,\nu\in \mathbb S^{n-1}$ and  $\nu\ne \pm \bar \nu $} \,.
    \end{displaymath}
    Hence, given $\adim = \vdim+1$,
    \begin{enumerate}[resume]
    \item if $F$ is a $\cnt^1$~integrand such that the corresponding
        function~$G$, as in~\eqref{eq:F-to-G}, is strictly convex, then $F \in
        \FAE_x(\mathcal P)$.
    \end{enumerate}
\end{remark}

\begin{remark}\label{remarkimpl}
    In~\cite[IV.1(7), p.~88]{Alm76} Almgren observes that uniformly convex
    functions give rise to anisotropic lagrangians satisfying $\FAUE_x(\mathcal
    P)$ in co-dimension $1$ and vice-versa, where $\mathcal P$ is the class of
    test pairs. Our result shows that functions that are just strictly convex
    give rise to anisotropic lagrangians satisfying $\FAE_x(\mathcal P)$ in
    co-dimension $1$, for every good family~$\mathcal P$. In~particular we
    deduce that there is no hope of improving Theorem \ref{thm:wBC-in-AE}
    showing that $\FwBC_x \subseteq \FAUE_x(\mathcal P)$. Indeed, if this was
    the case, in co-dimension one the strict convexity of the integrand would
    give rise to an anisotropic lagrangian satisfying $\FBC_x$ and consequently
    also $\FAUE_x(\mathcal P)$, which in turn would imply the uniform convexity
    of the integrand.
 \end{remark}

        





\subsection*{Acknowledgements}


The authors thank Guido De Philippis, Jonas Hirsch, Ulrich Menne and Robert
Young for useful discussions, in particular on
Sections~\ref{sec:wbc-in-ae}~and~\ref{sec:all-test-pairs-good}. The authors are
also grateful to Andrzej Weber without whose generous help the completion of
Section~\ref{sec:all-test-pairs-good} would not be possible.

The~first author has been supported by the NSF DMS Grant No.~1906451.
The~second author was supported by the National Science Centre Poland grant
no.~2016/23/D/ST1/01084.

\newcommand{\noopsort}[1]{}
\newcommand{\singleletter}[1]{#1}
\def\cprime{$'$}

{\small \noindent
Antonio De Rosa
\\
Courant Institute of Mathematical Sciences,
\\
New York University, New York, NY, USA
\\
\texttt{derosa@cims.nyu.edu}
}

\bigskip

{\small \noindent
  S{\l}awomir Kolasi{\'n}ski
  \\
  Instytut Matematyki,
  Uniwersytet Warszawski
  \\
  ul. Banacha 2, 02-097 Warszawa, Poland
  \\
  \texttt{s.kolasinski@mimuw.edu.pl}
}

\end{document}